\documentclass[11pt]{article}

\usepackage{graphicx,,psfrag}

\usepackage{amssymb,amsmath,amsthm,enumerate}
\usepackage{fullpage}
\usepackage{bbm}
\newcommand{\url}{}

\usepackage{tikz}

\newcommand\blfootnote[1]{%
  \begingroup
  \renewcommand\thefootnote{}\footnote{#1}%
  \addtocounter{footnote}{-1}%
  \endgroup
}

\RequirePackage[colorlinks,citecolor=blue,urlcolor=blue]{hyperref}

\newtheorem{lemma}{Lemma}
\newtheorem{remark}{Remark}
\newtheorem{proposition}{Proposition}
\newtheorem{definition}{Definition}
\newtheorem{theorem}{Theorem}

\newtheorem{corollary}{Corollary}

\newcommand{\tr}{{\rm tr}}

\newcommand{\dE}{\mathbb {E}}
\newcommand{\dP}{\mathbb{P}}
\newcommand{\dN}{\mathbb {N}}
\newcommand{\dR}{\mathbb {R}}
\newcommand{\dC}{\mathbb {C}}

\newcommand{\cW}{\mathcal {W}}
\newcommand{\cP}{\mathcal {P}}

\newcommand{\cL}{\mathcal {L}}

\newcommand{\SPAN}{ \mathrm{span}}

\newcommand{\ABS}[1]{{{\left| #1 \right|}}} 
\newcommand{\BRA}[1]{{{\left\{#1\right\}}}}
\newcommand{\NRM}[1]{{{\left\| #1\right\|}}} 
\newcommand{\NRMHS}[1]{\NRM{#1}_{\mathrm{{HS}}}} 
\newcommand{\PAR}[1]{{{\left(#1\right)}}} 

\newcommand{\AND}{\quad \mathrm{and} \quad}

\newcommand{\uP}{{\underline P}}

\newcommand{\cc}{\mathrm{cc}}

\newcommand{\1}{1\!\!{\sf I}}
\newcommand{\IND}{\1}
\newcommand{\veps}{\varepsilon}

\newcommand{\xbf}{\mathbf{x}}
\newcommand{\ybf}{\mathbf{y}}

\newcommand{\II}{I\hspace{-1pt}I}

\newcommand{\BEAS}{\begin{eqnarray*}}
\newcommand{\EEAS}{\end{eqnarray*}}
\newcommand{\BEA}{\begin{eqnarray}}
\newcommand{\EEA}{\end{eqnarray}}
\newcommand{\BEQ}{\begin{equation}}
\newcommand{\EEQ}{\end{equation}}
\newcommand{\BIT}{\begin{itemize}}
\newcommand{\EIT}{\end{itemize}}
\newcommand{\BNUM}{\begin{enumerate}}
\newcommand{\ENUM}{\end{enumerate}}

\definecolor{darkred}{rgb}{0.9,0,0.3}

\title{Spectral gap of sparse bistochastic matrices with exchangeable rows} 

\author{Charles Bordenave, Yanqi Qiu and Yiwei Zhang\blfootnote{CB is supported by French ANR grant ANR-14-CE25-0014 and ANR-16-CE40-0024-01. YQ is supported by  National Natural Science Foundation of China grants NSFC Y7116335K1 and NSFC 11688101.YZ is supported by National Science Foundation of China grant NSFC 11701200, NSFC 11871262, and AMS China exchange program KY and Yu-Fen Fan fund travel grant.}}

\begin{document}
\maketitle

\abstract{We consider a random bistochastic matrix of size $n$ of the form $M Q$ where $M$ is a uniformly distributed permutation matrix and $Q$ is a given bistochastic matrix. Under sparsity and regularity assumptions on $Q$, we prove that the second largest eigenvalue of $MQ$ is essentially bounded by the normalized Hilbert-Schmidt norm of $Q$ when $n$ grows large. We apply this result to random walks on random regular digraphs.} 

\section{Introduction}

\subsection{Model and main result}
For $n\geq 1$ integer,  let $[n] = \{ 1, \cdots, n \}$.  Let $Q \in M_n(\dC)$ be a bistochastic matrix of size $n$, that is, for any $x,y$ in $[n]$, $Q_{xy} \geq 0$ and the constant vector $\IND = (1, \cdots, 1)   \in \dR^n$ is an eigenvector of $Q$ and its transpose $Q^\intercal$:
\begin{equation}\label{eq:PerronQ}
Q \IND = Q^ \intercal \IND = \IND.
\end{equation}
In probabilistic terms, $Q$ is the transition matrix of a Markov chain on $[n]$ which admits the uniform measure as an invariant measure.

Let $\mathbb{S}_n$ be the symmetric group on $n$ elements. We will denote by $| \cdot |$ the cardinal number of a set and the usual absolute value, $\dP(\cdot)$ and $\dE (\cdot)$ are the probability and expectation under the uniform measure on $\mathbb S_n$: for any subset $E \subset \mathbb S_n$,
$$
\dP ( E )  = \frac{ | E |}{ |\mathbb S_n|}.
$$

Let $\sigma$ be a uniformly distributed random permutation in $\mathbb S_n$. We denote by $M$ the $n \times n$ permutation matrix of $\sigma$. In matrix notation, for all $x,y \in [n]$,
$$
M_{xy} = \mathbbm{1} (\sigma(x) = y).
$$

In this paper, we study the $n \times n$ random matrix
\begin{equation}\label{eq:defP}
P  =  M  Q.
\end{equation}
or, in matrix notation, for all $x,y \in [n]$,
$
P_{xy} = Q_{\sigma(x) y}.
$
Then, $P$ is the transition matrix of a Markov chain on $[n]$ where at each step, we compose with $\sigma$ before performing a step according to $Q$. Note that $P$ itself is bistochastic and thus the constant vector $\IND$ is an eigenvector of $P$ and its transpose $P^\intercal$ with eigenvalue $1$. From Perron-Frobenius theorem, it follows that $1$ is the largest eigenvalue of $P$. We order non-increasingly the moduli of the eigenvalues of $P$, $\lambda_i = \lambda_i(P)$,
\begin{equation}\label{eq:eigP}
1 = \lambda_1 \geq |\lambda_2| \geq \cdots \geq |\lambda_n|.
\end{equation}

The spectral gap is defined as $1 - |\lambda_2|$. It measures the asymptotic mixing rate to equilibrium. For example, if $P$ is aperiodic and irreducible, then for any probability measure $\pi_0$ on $[n]$,
$$
\lim_{t \to \infty}  \NRM{\pi_0 P^t  - \pi }_{\mathrm{TV}}^{1/t} = |\lambda_2|.
$$
where $\pi  = \IND / n$ is the invariant measure of $P$ and, for a signed measure $\nu$ on $[n]$,  $\NRM{ \nu  }_{\mathrm{TV}} = \frac 1 2 \sum_{x} |\nu (x)|$ denotes the total variation norm (we refer to \cite{MR3726904}).

Our main result is a sharp probabilistic upper bound on $|\lambda_2|$ which involves strikingly very few parameters of $Q$. For $A \in M_n( \dC)$, the normalized Hilbert-Schmidt norm is defined as
\begin{equation}\label{eq:defHS}
\NRMHS{ A } = \sqrt{ \frac 1 n \tr ( AA^*) } = \sqrt{ \frac 1 n \sum_{x,y} |A_{xy}|^2 } =  \sqrt{ \frac 1 n \sum_{i=1}^n s_i (A)^2},
\end{equation}
where the scalars $s_i(A)$, denote the singular values of $A$ (that is, the eigenvalues of $\sqrt{A A^\intercal}$ and $\sqrt{A^\intercal A}$).

The $\ell^1$ to $\ell^\infty$ norm of $A \in M_n( \dC)$ is
$$
\NRM{ A }_{1 \to \infty} = \max_{x,y} |A_{xy}|.
$$
For some applications, we introduce a relaxation of this norm. It is  defined, for  $0 < \delta \leq 1$, as
\begin{equation}\label{eq:defAd}
\NRM{ A }^{(\delta)}_{1 \to \infty} = \inf_{ \mathcal{E} \subset [n], |\mathcal{E}| < n^{1-\delta} }\max_{x  \notin \mathcal{E} ,y} |A_{yx}|,
\end{equation}
(note that this is not a norm for $\delta \ne 1$ and  $\NRM{ A }^{(1)}_{1 \to \infty} = \NRM{ A }_{1 \to \infty}$). We also introduce a usual sparsity parameter of $A \in M_n( \dC)$, defined as
\begin{equation}\label{eq:defd}
\NRM{ A }_{1 \to 0} = \max_{ x }  | \{ y : A_{xy} \ne 0 \} |,
\end{equation}
(this is the $\ell^1$ to $\ell^0$ pseudo-norm for the pseudo-norm $\ell^0$ on $\dC^n$, $\| u \|_{\ell^0} = \sum_x \IND ( u_x \ne 0)$).

For the remainder of the text, we fix some $0 < \delta <1$ and set the following  notation
$$
d:= \NRM{ Q^\intercal Q}_{1 \to 0} \AND \rho := \NRMHS{ Q } \vee \NRM{ Q }_{1 \to \infty}^{(\delta)}.
$$

We will always assume that $d \geq 2$ (otherwise $d = 1$, $Q$ itself is a permutation matrix and $P$ and $M$ have the same distribution). We observe that $d$ and $\rho$ are intrinsic parameters of $P$ since
$
\NRMHS{ Q } = \NRMHS{ P }$, $\NRM{ Q}^{(\delta)}_{1 \to \infty} = \NRM{ P }^{(\delta)}_{1 \to \infty}$, $\NRM{ Q^\intercal Q}_{1 \to 0} = \NRM{ P^\intercal P }_{1 \to 0}.
$
Note also that the singular values of $P$ and $Q$ are equal.  Our main result  asserts that $|\lambda_2|$ is essentially bounded by $\rho$ as long as $d$ is not too large.
\begin{theorem}\label{th:main}
Let $n \geq 1$ be an integer and let $\sigma$ be a uniformly distributed random permutation in $\mathbb S_n$. Let $M$ be the permutation matrix of $\sigma$ and $Q \in M_n(\dR)$ be a bistochastic matrix as above. Let $P = MQ$ whose eigenvalues are denoted as in \eqref{eq:eigP}. For any $0 < c_0 <  \delta \leq 1$, there exists a  constant $ c_1 > 0 $ (depending only on $\delta,c_0$) such that
$$
\dP \PAR{ | \lambda_2 |  \geq  (1  + \veps) \rho }  \leq n^{-c_0},
$$
where $$\veps =c_1 \frac{\log d }{\sqrt{ \log n}}.$$
\end{theorem}

See Figure \ref{fig:sim} for numerical simulations. Theorem \ref{th:main} implies that in many cases, the second largest eigenvalue of $P$ is much smaller than the second largest eigenvalue of $Q$. Assume for example that $Q$ is symmetric (in probabilistic term, $Q$ is a reversible Markov chain) and that  $\rho = \NRMHS{ Q }$ (that is $\NRM{ Q  }_{1 \to \infty}^{(\delta)} \leq \NRMHS{ Q }$). Then the eigenvalues of $Q$ are real and their absolute values coincide with the singular values of $Q$. From \eqref{eq:defHS}, $\NRMHS{ Q }$ is the $\ell ^2$-average of the eigenvalues of $Q$, the latter is typically much smaller than the  second largest eigenvalue of $Q$ in absolute value.
Note also that the eigenvalues of $M$ are all of modulus $1$ and that, with probability tending to $1$ as $n$ goes to infinity, $M$ is non irreducible. It follows that even if the Markov chains $Q$ and $M$ have a small spectral gap ($Q$ may even be non irreducible) then the composed Markov chain $ P = MQ$ has typically a large spectral gap.

 \begin{figure}[h]
\begin{center}
\includegraphics[angle =0,height = 6cm]{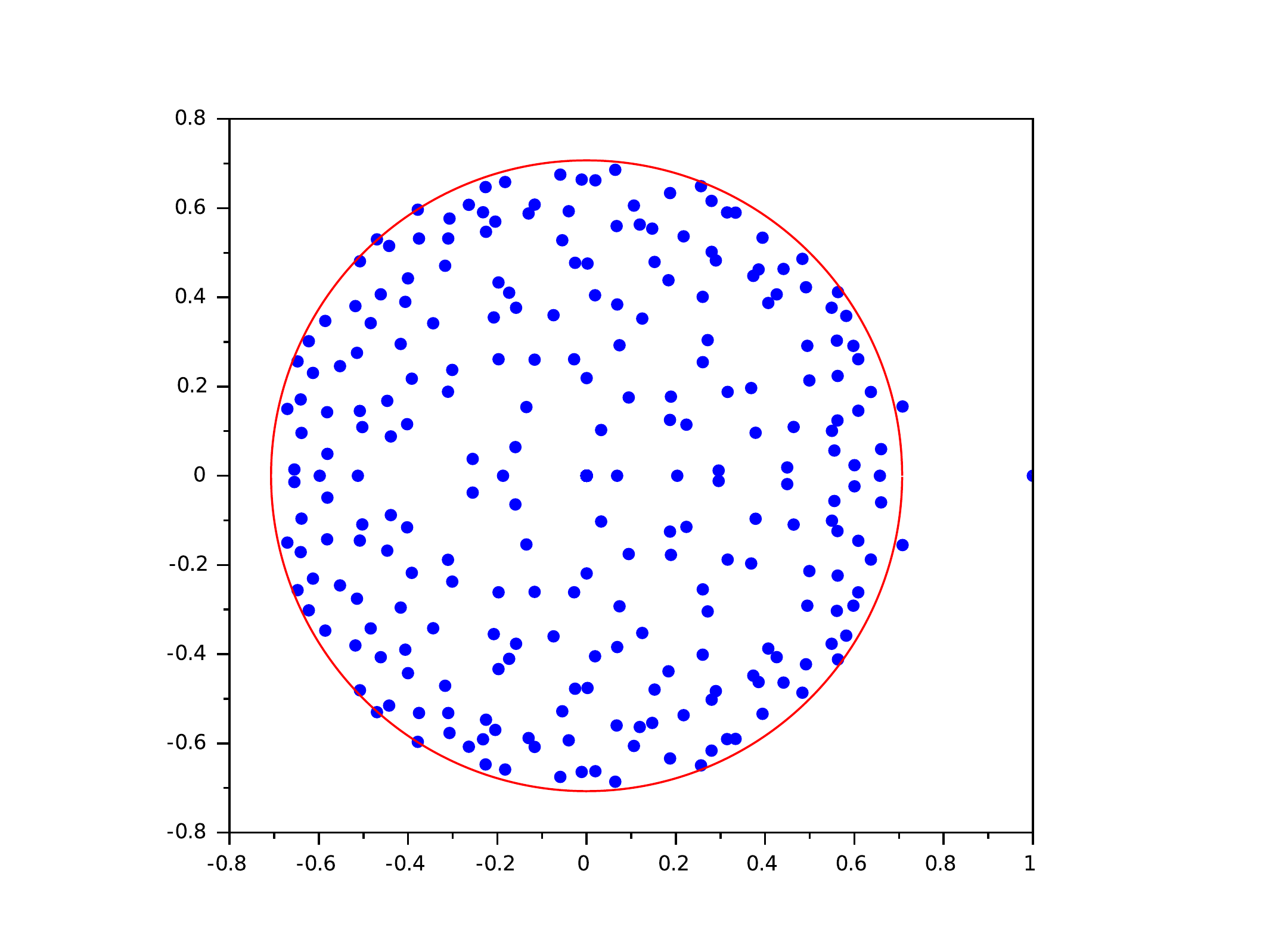}\quad
\includegraphics[angle =0,height = 6cm]{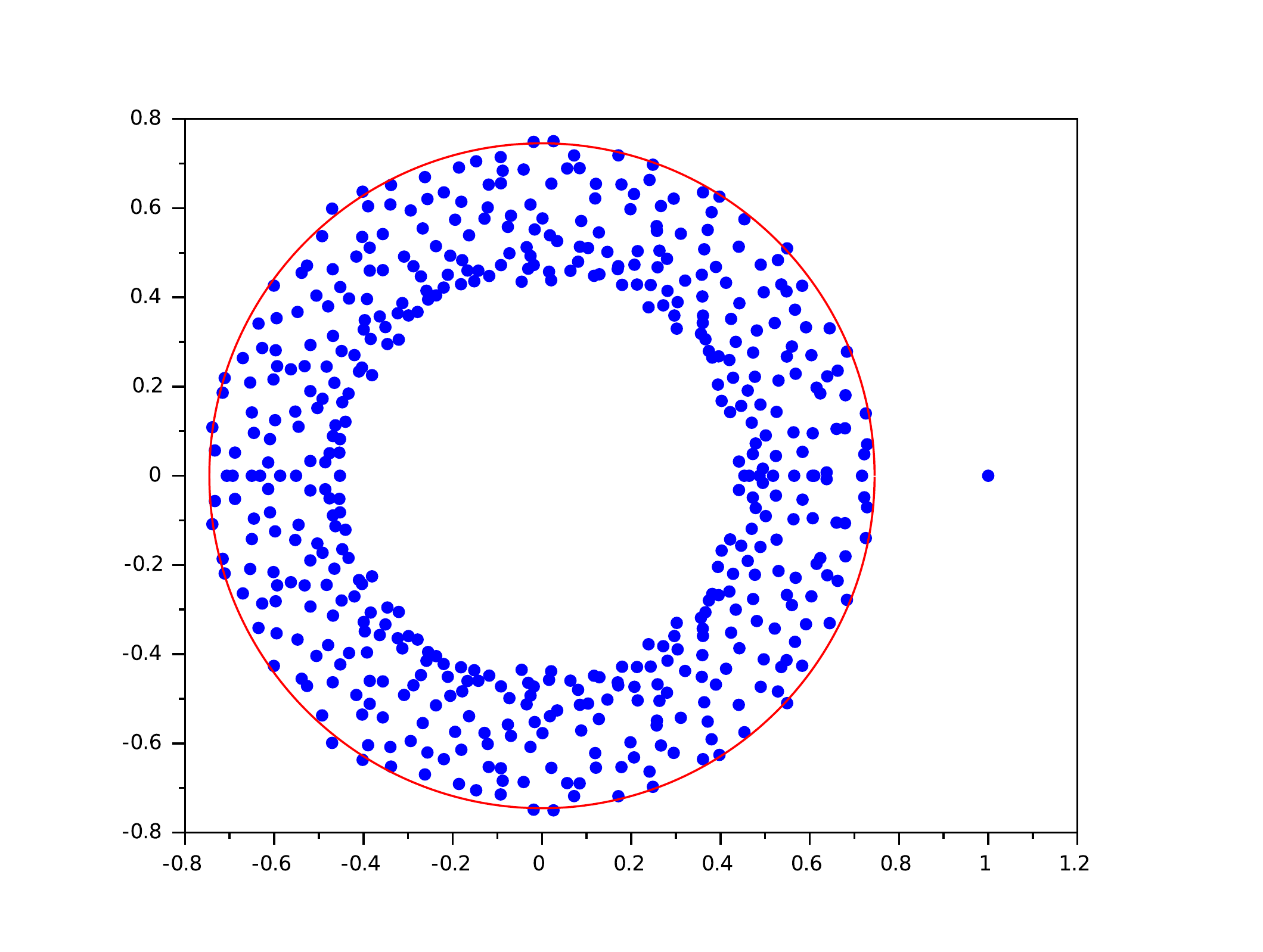}
\caption{Plot of the eigenvalues of $P$ for a single realization of $M$ when $n = 500$ and $Q = p I_n + (1 -p) I_{n/2} \otimes D$ where $I_n$ is the identity matrix of size $n$, $D$ is the matrix of size $2$ given by $D_{11}= D_{22} = 0$, $D_{21} = D_{12} = 1$ with $p= 1/2$ (left) and $p = 1/3$ (right). The circles in red have radii $\NRMHS{Q} = \sqrt{p^2 + (1-p)^ 2}$.}.
\label{fig:sim}
\end{center}
\end{figure}

The conclusion of Theorem \ref{th:main} is especially interesting when $\rho = \NRMHS{ Q }$. This is a condition on the inhomogeneity of the matrix $Q$. Indeed, observe that
\begin{equation*}\label{eq:gtgt}
\max_{y} Q_{yx} \leq \sqrt{\sum_y Q_{yx} ^2 }.
\end{equation*}
Assume that the right-hand side of the above inequality does not depend on $x$. Then we find that $\NRM{ Q }_{1 \to \infty} \leq \NRMHS{ Q }$ and $\rho =\NRMHS{ Q }$. The latter condition holds for example if $Q$ is a transition matrix of simple random walk on the simple regular graph.

We remark that the order $n^{-c_0}$  in Theorem \ref{th:main} cannot be improved  significantly when $Q$ admits an invariant subspace of small dimension spanned by vectors of the canonical basis $(e_x)_{x \in [n]}$.  More precisely, assume for example that $H = \SPAN( e_1, \ldots, e_k)$ is the invariant subspace of $Q$ for some fixed integer $1 \leq k \leq n/2$.  Consider the event $\sigma([k]) = [k]$. It is not hard to check that this event has probability $1/ {n \choose k} \geq 1/n^{k}$. On this event, $H$ and its orthogonal $H^\perp$ are both  invariant by $Q$. Hence, on this event, $\lambda_1 = \lambda_2 = 1$ and
$$
\dP \PAR{ | \lambda_2 |  = 1 }  \geq n^{-k}.
$$

Similarly, if $\delta =0$ (that is, $\rho = \NRMHS{ Q }$), the conclusion of the theorem may be wrong. Assume for example that $Q$ is a bistochastic matrix such that the subset
$$
S = \{x \in [n]: \text{$Q_{y_x x } =1$ for some $y_x \in [n]$}\}
$$
is of positive proportion in $[n]$. Then the probability that for at least one of such $x \in S$, we have $\sigma(x) = y_x$ is uniformly lower bounded in $n$. On the latter event, $\lambda_2 = 1$ since $P_{xx} = 1$.

We expect that when $\rho = \NRMHS{ Q }$ and $d = \exp(o(\sqrt{\log n}))$, the conclusion of Theorem \ref{th:main} is sharp. Namely, we conjecture that for any $\veps > 0$, $ |\lambda_2| \geq  (1 - \veps) \rho$ with probability tending to $1$ as $n$ goes to infinity.  In the next subsection, we will discuss some examples where the conjecture is true. There is an indirect evidence supporting this conjecture when we replace the random permutation matrices by other random unitary matrices. Let $U$ be a random unitary matrix of size $n$ sampled according to the Haar measure on the unitary group. Under mild assumptions on $Q$, it is known that the  spectral radius of $U Q / \NRMHS{ Q }$ converge in probability to $1$, see \cite{MR2831116,MR3000558,MR3164983} and, for the connection to free probability \cite{MR1784419,MR3585560}. More generally, from these references, we might also guess an asymptotic formula for the empirical distribution of the eigenvalues of $P/\NRMHS{ Q}$.

Theorem \ref{th:main} is related to the recent work by Coste \cite{simon}. There, the author studies the spectral gap of the transition matrix of simple random walk on a random digraph. With our notation, it corresponds to  the second eigenvalue of a Markovian matrix of size $m$, proportional to $n$, of the form $A S B^\intercal$, where $S$ is uniformly distributed in $\mathbb S_n$ and $A,B$ are specific matrices in $M_{m,n} (\dC)$ such that $A\IND_n = \IND_m$ and $B^\intercal \IND_m = \IND_n$. In some cases treated in \cite{simon}, the upper bound on $|\lambda_2|$ is also given by $(1+o(1)) \NRMHS{B^\intercal A}$ . Our two results are thus of the same nature even if they are not directly comparable.

We remark finally that Theorem \ref{th:main} can be extended to some extend beyond the uniform measure on $\mathbb S_n$, see Remark \ref{rk:bunif} below, and beyond bistochastic matrices, see Remark \ref{rk:bbis} (for examples to matrices $Q$ such that $\IND$ is a common eigenvector of $Q$ and $Q^\intercal$).

\subsection{Random walks on random digraphs}

In this section, we state some immediate consequences of Theorem \ref{th:main}.

A \emph{digraph} $G = (V,E)$ is the pair formed by a countable vertex set $V$ and a set of oriented edges $E \subset V \times V $. If $e = (u,v) \in E$ then $e$ is an incoming edge of $v$ and an outgoing edge of $u$. For $r \in \dN$, we say that $G$ is $r$-regular if any vertex has exactly $r$ incoming and $r$ outgoing edges. If the set $E$ is symmetric then $G$ can be interpreted as an undirected graph.

\begin{theorem}
\label{th:RWDRW} Let $n\geq 1$ and $r \geq 2$ be  integers and  $Q$ be the transition matrix of a simple random walk on a $r$-regular digraph $G = (V,E)$ with $V = [n]$. Let $\sigma$ be a uniformly distributed permutation in $\mathbb S_n$ and let $M$ be its permutation matrix. Let $P = MQ$ be as in \eqref{eq:defP} with eigenvalue denoted as in \eqref{eq:eigP}. For any $0 < c_0< 1$, there exists $c_1>0$ (depending only on $c_0$) such that the conclusion of Theorem \ref{th:main} holds with $\rho = 1/ \sqrt{r}$ and $d = r^2$.
\end{theorem}

In the above theorem, the matrix $P$ is the transition matrix of the simple random walk on the random digraph $G^\sigma = (V, E^\sigma)$  where $E^\sigma = \{ ( \sigma^{-1} (x), x') : (x,x') \in E \}$. Note that $G^\sigma$ will have many weak cycles of length $4$ if $G$ has many weak cycles of length $4$.

Theorem \ref{th:RWDRW} can be applied to uniformly sampled $r$-regular digraphs.

\begin{corollary}\label{cor:SRDG}
Let $n \geq 1$ and $r \geq 2$ be integers. Let $P$ be sampled uniformly over bistochastic matrices of size $n\times n$ with entries in $\{0,1/r\}$ and with eigenvalues as in \eqref{eq:eigP}.  Then for any $0 < c_0< 1$, there exists $c_1>0$ (depending only on $c_0$) such that the conclusion of Theorem \ref{th:main} holds with $\rho = 1/ \sqrt{r}$ and $d = r^2$.
\end{corollary}

For  $r \geq 2$ uniformly bounded in $n$, Corollary \ref{cor:SRDG} is contained in \cite[Corollary 1.2]{simon}.  There is a converse of Corollary \ref{cor:SRDG} in some range of  the degree $r$. It is a consequence of the main results in \cite{Cook,LLTTY} that,  if $r \leq n  - (\log n) ^{96}$ and $r \to \infty$,  then, for any $\veps >0$, with probability tending to $1$ as $n$ goes to $\infty$,  $ |\lambda_2| \geq (1-\veps) \rho$.  Hence, if $r \to \infty$ and $r = \exp ( o (\sqrt{\log n}))$, $|\lambda_2| / \rho$ converges in probability to $1$ as $n\to \infty$.

Let us give another application of Theorem \ref{th:main}. From Birkhoff-von Neumann Theorem, the set of bistochastic matrices is the convex hull of permutation matrices. We thus have the decomposition
\begin{equation}\label{eq:BvN}
Q  = \sum_{i=1}^r p_i M_i,
\end{equation}
where $M_i$ are permutations matrices and $(p_1, \cdots, p_r)$ is a probability vector. This decomposition is not unique in general. Our next result asserts that if $Q$ admits such decomposition with $r$ not too large  and matrices $M_i$ which have few common non-zeros entries then the second largest eigenvalues of $P$ is at most $(1+ o(1) ) \sqrt{\sum_i p _i ^2}$.

\begin{theorem} \label{th:wperm}
Let $n\geq 1$ and $r \geq 2$ be  integers, $p = (p_1, \ldots , p_r)$ be a probability vector and $\sigma_1, \ldots, \sigma_r$ be permutations in $\mathbb S_n$ with associated permutation matrices $M_1, \ldots , M_r$. Assume that $Q$ is given by \eqref{eq:BvN}. We set $S = \{ x \in [n]  : \exists i \ne j , \sigma_i (x) = \sigma_j (x) \}$.  Let $\sigma$ be a uniformly distributed permutation in $\mathbb S_n$ and let $M$ be its permutation matrix. Let $P = MQ$ be as in \eqref{eq:defP} with eigenvalue denoted as in \eqref{eq:eigP}. For any $0 < c_0 < \delta \leq 1$, there exists a  constant $ c_1 > 0 $ (depending only on $\delta,c_0$) such that if $|S| \leq n^{ 1 -\delta }$,  then the conclusion of Theorem \ref{th:main} holds with $\rho = \sqrt{ \sum_i p_i^2 }$ and $d = r^2$.
\end{theorem}

In Theorem \ref{th:wperm}, assume that $S = \emptyset$. Then $G = (V,E)$ and $G^{\sigma} = (V, E^{\sigma})$ with $V = [n]$, $E = \{ (x, \sigma_i(x) ) : x \in V, i \in [r] \}$ and $E^\sigma = \{ ( \sigma^{-1} (x), \sigma_i (x)) : x \in V , i \in [r] \}$ are $r$-regular digraphs. The transition matrices $Q$ and $P$ correspond to anisotropic random walks on $G$ and $G^\sigma$. Interestingly, the scalar $\sqrt{\sum_i p_i ^2}$ is the spectral radius of the anisotropic random walk on the infinite homogeneous directed tree, see  the monograph \cite{MR1219707}.

\begin{corollary}\label{cor:ani} Let $n\geq 1$ and $r \geq 2$ be  integers, $p = (p_1, \ldots , p_r)$ be a probability vector and $\sigma_1, \ldots, \sigma_r$ be independent and uniformly distributed permutations in $\mathbb S_n$ with associated permutation matrices $M_1, \ldots , M_r$. Set
$$
P = \sum_{i=1} ^r p_i M_i
$$
with eigenvalue denoted as in \eqref{eq:eigP}. For any $0 < c_0 < 1$, there exists a  constant $ c_1 > 0 $ (depending only on $c_0$) such that the conclusion of Theorem \ref{th:main} holds with $\rho = \sqrt{ \sum_i p_i^2 }$ and $d = r^2$.
\end{corollary}

Consider the setting of Corollary \ref{cor:ani} in the case $p_i = 1/r$ for all $i \in [r]$. Then $\rho = 1/\sqrt r$. It follows from the main result in \cite{BCZ} that if, for some $c >0$, $(\log n)^{12} \leq r \leq c n$ then for any $\veps >0$, $|\lambda_2| \geq (1 - \veps)\rho$ with probability tending to $1$ as $n$ goes to infinity. Hence, in the regime $(\log n)^{12} \leq r \leq \exp (  o ( \sqrt{\log n}))$, $|\lambda_2| / \rho$ converges in probability to $1$ as $n$ goes to infinity.

\subsection{Fluid mixing protocol driven by shuffling-and-fold maps}
In this section, we present a physical interpretation of our main theorems in the setting of fluid mechanical kinematics. Let us briefly state a background of this subject. Generally speaking, the motion of fluid particles is described with a map $S:\mathcal{R}\to S(\mathcal{R})$, where $\mathcal{R}$ refers to fluid particles, and $S(\mathcal{R})$ refers to one advection cycle. Similarly, $n$ advection cycles are obtained by $n$ repeated application of $S$, and denote by $S^{n}(\mathcal{R})$. Meanwhile, put a probability measure $\mu$ that assigns to any (mathematically well-behavior) subdomain of $\mathcal{R}$ as its volume. The incompressibility of the fluid is expressed by stating that, as any subdomain $A\subset\mathcal{R}$ is stirred, $\mu(A)=\mu(S^{-1}(A))$, i.e., the volume of $A$ is preserved under the application of $S$. The definition of $S$ is \emph{mixing} is that:
\begin{equation}\label{equ:mixing}
  \lim_{n\to\infty}\mu(S^{-1}(A)\cap B)=\mu(A)\cdot\mu(B),
\end{equation}
for all Borel subsets $A,B$ of $\mathcal{R}$. This states that under the action of advection cycle on $A$, one expect to find the same amount of $A$ in any of the chosen $B$. Equation \eqref{equ:mixing} can be reformulated in functional form as the action of $S$ on observations $g$ and $h$ via \emph{the decay of correlations}:
\begin{equation}\label{equ:decayofcorrelation}
  \mathcal{C}_{g,h}(n):=\left|\int h(g\circ S^{-n})d\mu-\int g d\mu\cdot\int h d\mu \right|\to 0,~~\mbox{as}~~n\to\infty.
\end{equation}
The observation $g$ and $h$ are representative of scale field with certain regularity. Of course, once $S$ is mixing, the rate of $\mathcal{C}_{g,h}$ gives a quantifier of the speed of mixing. We refer to the book \cite{SOW06} and two recent surveys \cite{Aref04,FGW16} from either physical or mathematical detailed explanations respectively.

Good mixing protocol can be accomplished by the action of stretch and fold (SF) elements, though a cascade to small scales via turbulent eddies \cite{Aref84}. The SF property has been extensively characterized by the uniformly expanding property in the language of dynamical systems. A transfer operator $\mathcal{L}_{f}$ can be associated by an smooth uniformly expanding map $f$, with
\begin{equation}\label{equ_transfer}
  (\cL_{f}\phi)(x):=\sum_{y:f(y)=x}\frac{\phi(y)}{det|Df(y)|}.
\end{equation}
The uniform expanding property ensures $det|Df(y)|\neq 0$ for every point $y$. Then, there is an absolutely continuous invariant probability measure $\mu$ with the density $0<\frac{d\mu}{dLeb}<+\infty$ being the fixed point of $\mathcal{L}_{f}$, and various functional spaces $\Upsilon$ containing smooth observations have been verified preserved by $\mathcal{L}_{f}$, and moreover $\mathcal{L}_{f}$ is (quasi)-compact on $\Upsilon$, i.e., $r_{e}(ess_{\Upsilon}(\mathcal{L}_{f}))<1$, where $r_{e}$ is the spectral radius of the essential spectrum \footnote{A complex number $\lambda$ belongs to $\mbox{ess}(\cL_{f})$, if $\lambda$ is the limit point of $\mbox{spec}(\cL_{f})$. Therefore, the essential spectrum $\mbox{ess}(\cL_{f})$ is a closed set, and $\mbox{spec}(\cL_{f})\backslash\mbox{ess}(\cL_{f})$ consists of at most countably many isolated points which have no limit points outside $\mbox{ess}(\cL_{f})$.} and $1$ is the spectrum radius on $\Upsilon$, (see \cite{Bal00} for the detailed proof on these assertions). Under this setting, the decay of correlation (for observations in $\Upsilon$) shrinks exponentially, with the optimal rate
\begin{align*}
r_{e}(ess_{\Upsilon}(\mathcal{L}_{f}))\leq \tau^{\Upsilon}_{f}&=\inf\{\tau:\mathcal{C}_{g,h}(n)<const_{g,h}\cdot\tau^{n},~~\forall g,h\in\Upsilon,\forall n\in\mathbb{N}\}\\
&=\sup\{\rho:\rho\in\mbox{Spec}_{\Upsilon}(\mathcal{L}_{f})\backslash\{1\}\}<1.
\end{align*}
That is, the maximum of the essential spectrum radius and subdominant eigenvalue of $\mathcal{L}_{f}$ fully determines the mixing rate.

Meanwhile, another mixing process cutting and shuffling (CS), which can increase the number of interfaces and segregation, but doesn't involve material deformation, naturally arise in many circumstances. For instances, split and recombine micromixers adopt the action of CS to increase the number of lamellae between substances \cite{HM98}; Streamline jumping occurs during reorientation and creates pseudoelliptic and pseudohyperbolic period points \cite{LRMTOH10,SRLM16}; High strain in polymeric with shear banding cause slip deformations \cite{LLC12}. All of these mixing protocols exhibit a combinational mechanisms of both SF and CS.

Under this framework, several authors considered the composition of a permutations of equal size cells, or more generally a piecewise isometries $\bar{\sigma}$ with a piecewise expanding maps $f$, and study how the correspond optimal mixing rate $\tau_{f\circ\bar{\sigma}}$ varies with respect to the different choices of $\bar{\sigma}$. In fact, a better understanding of such effects would be expected to deepen our knowledge on the balance between global transporting rate and local diffusivity \cite{FHW02,KCOL12,KSW17,WV02,TC03}.

We will particularly concentrate on the toy model as follows. Let $f(x):=rx \mod 1$ on the torus $[0,1)$ with $r\in\mathbb{N}$. On the other hand, to any permutation $\sigma\in \mathbb S_{n}$,  we associate a linear map, denoted by $\bar{\sigma}$ defined for $i \in [n]$ by
\begin{equation}\label{equ_permutation}
  \bar{\sigma}(x):=x+\frac{\sigma(i)-i}{n}, ~~ \forall x \in I_i:=[\frac{i-1}{n},\frac{i}{n}).
\end{equation}
We are interested in linear expanding maps of the form $f \circ \bar \sigma$, see Figure \ref{fig:2} for an example. This combination model was first introduced in \cite{MR3021362}, and could be used as the basis for study the two dimensional Baker's map composing with CS behavior on its domain \cite{KSW17}. Interestingly, composition of permutations do not improve mixing rate, and typically make it worse. This is contrast to the model considering by Ashwin.et.al \cite{ANK02}, where combining permutations with diffusion from a Gaussian heat kernel accelerates the mixing rate.
\begin{figure}[h]
\begin{center}
\includegraphics[angle =0,height = 4cm]{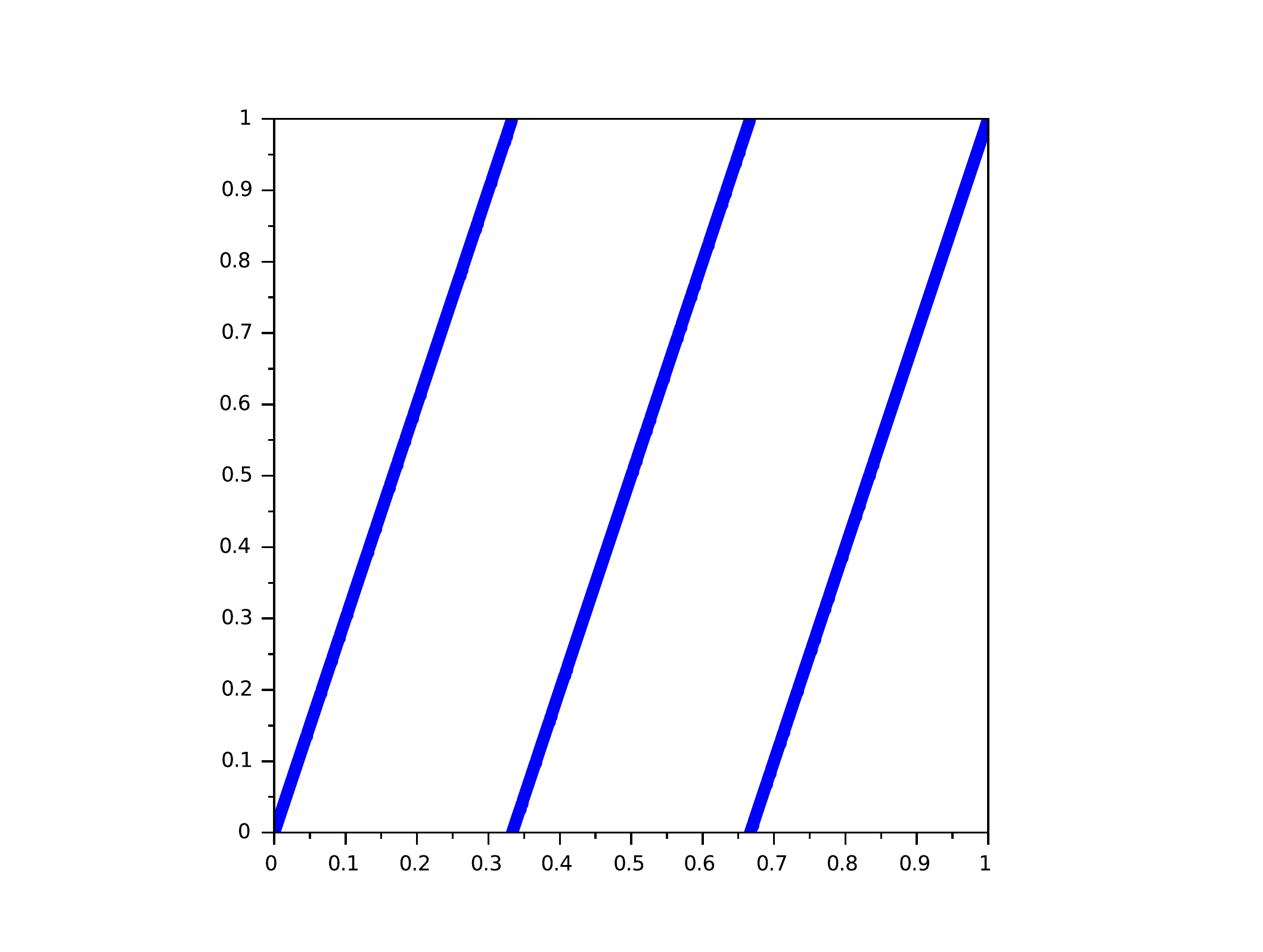}
\includegraphics[angle =0,height = 4cm]{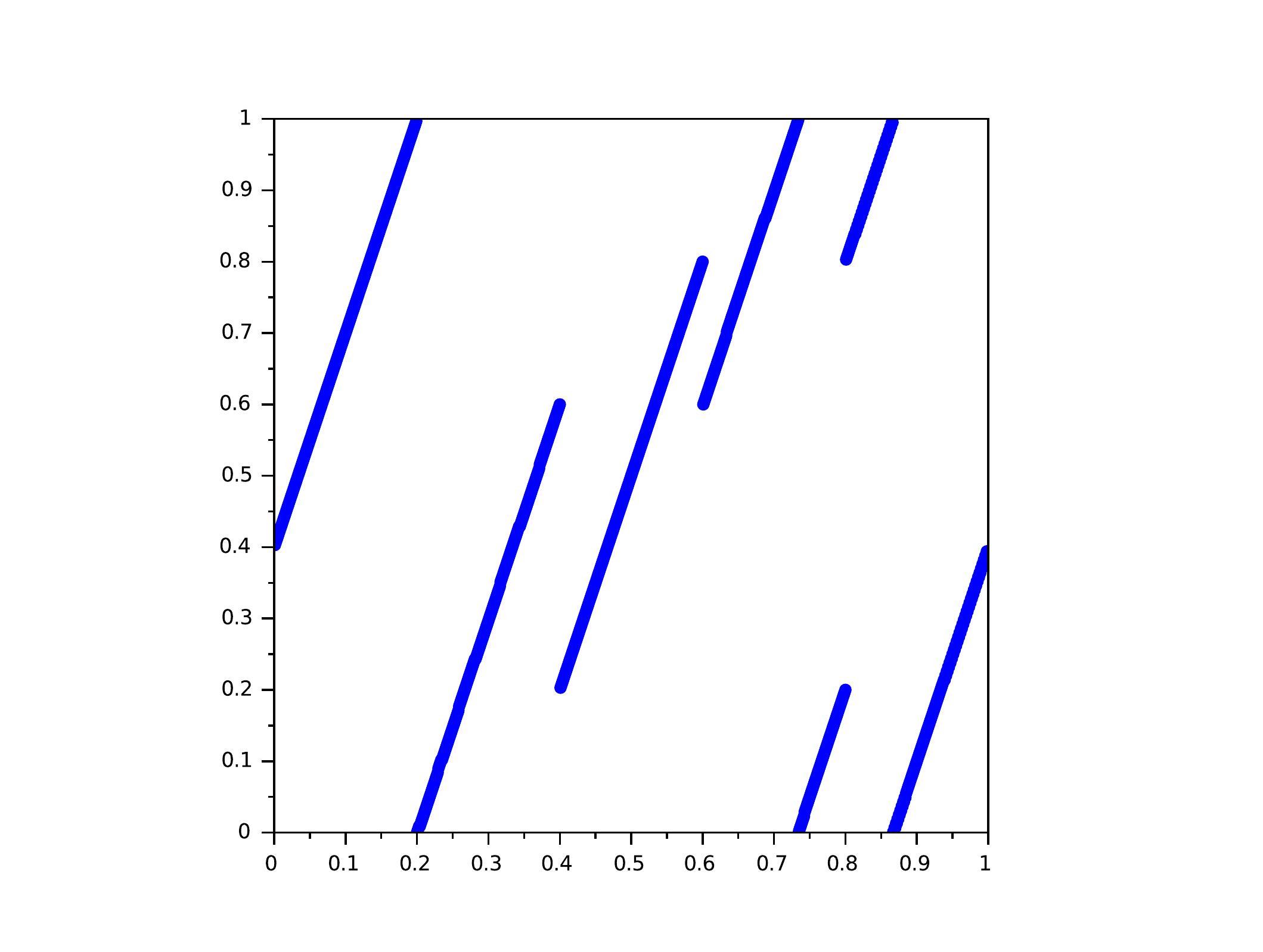}
\includegraphics[angle =0,height = 4cm]{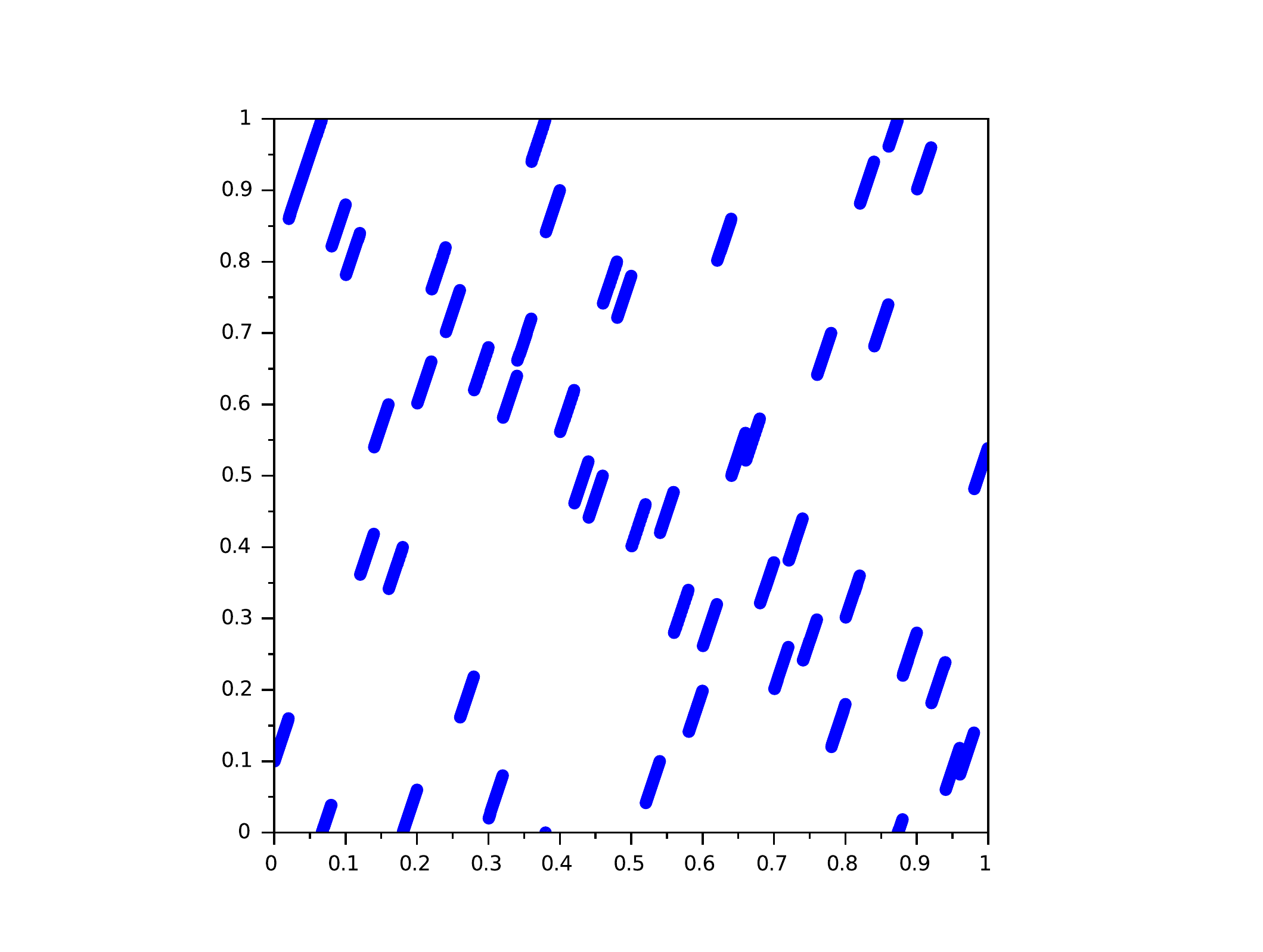}
\caption{Plot of the SF maps $f(x)=3 x\mod 1$ (left) and $f \circ \bar \sigma$ for $n= 5$ with $\sigma = (5  \, 1  \,   3   \,  2  \,  4 )$ (middle) and $n = 50$ with $\sigma$ uniformly distributed (right).}
\label{fig:2}
\end{center}
\end{figure}

Based on the construction, for each permutation $\sigma$,
$$
\mathcal{L}_{f\circ\bar{\sigma}}\varphi(x)=\frac{1}{r}\cdot\sum_{y:f\circ\bar{\sigma}(y)=x}\varphi(y),
$$
and $\mathcal{L}_{f\circ\bar{\sigma}}\mathbf{1}=\mathbf{1}$, where $\mathbf{1}$ is the constant function on $[0,1]$. Thus, the Lebegue measure itself is preserved by $f\circ\bar{\sigma}$.

There is a standard way to reduce the mixing rate $\tau_{f\circ\bar{\sigma}}$ estimation into finite dimensional matrices' eigenvalue estimation (e.g. see \cite[Chapter 9]{BoGo97} for detailed explains). For each $n\geq r$, we define the Markov transition matrix, say $Q^{(\bar{\sigma})}$ of $f\circ\bar{\sigma}$, by for all $i,j\in [n]$,
\begin{equation}\label{equ:transitionmatrix}
Q^{\bar{\sigma}}_{ij}:=\left\{\begin{array}{ll}
                     1/r, & \mbox{if}~ ((f\circ\bar{\sigma})^{-1}(I_{j}))\cap I_{i}\neq\emptyset, \\
                     0, & \mbox{Otherwise}.
                   \end{array}\right.
\end{equation}
It is straightforward to see that $Q^{\bar{id}}$ is a bistochastic matrix, and $Q^{\bar{\sigma}}=M\cdot Q^{\bar{id}}$, where $M$ is the permutation matrix for $\sigma$. Thus $Q^{\bar{\sigma}}$ is a bistochastic matrix for every permutation $\sigma$. On the other hand, by checking the Lasota-Yorke inequality \cite{MR0335758}, it has been verified that when the functional space $\Upsilon$ is chosen from either $\mathcal{A}$, the space of bound holomorphic complex valued functions on $[0,1)$ with continuous extension to the boundary; $\mathcal{C}^{k}$ the space of complex valued functions on $[0,1)$ has $k$-th continuous derivatives; or $BV$, the space of complex valued functions of bounded variation, such that the transfer operator $\mathcal{L}_{f\circ \bar{\sigma}}$ on $\Upsilon$ is (quasi)-compact. Moreover,  Mayer\cite{MR576928}, Ruelle\cite{MR1920859}, Keller\cite{MR1005524} et.al, developed the dynamical Fredholm theory method of Markov shifts which  indicates that all the isolated eigenvalue $\rho$ for $\mathcal{L}_{f\circ\bar{\sigma}}$ is an eigenvalue of $Q^{\bar{\sigma}}$ (e.g. \cite[Theorems 2.7]{Bal00} for analytic case; and \cite[Theorem 2.9]{Bal00} for $\mathcal{C}^{k}$ case; and \cite[Theorem A]{Mori90} for BV case) on $f\circ\bar{\sigma}$). That is to say, for every permutation $\sigma$,
\begin{equation}\label{equ:mixingrate}
  \tau^{\Upsilon}_{f\circ\bar{\sigma}}=\max\{r_{e}(ess_{\Upsilon}\mathcal{L}_{f\circ\bar{\sigma}}),|\lambda_{2}(Q^{(\bar{\sigma})})|\}.
\end{equation}
Meanwhile, their dynamical Fredholm theory method also indicates that the exact value of the essential spectrums for every permutation $\sigma$ can be estimated by
\begin{equation}\label{equ:exactvalue}
r_{e}(ess_{\mathcal{A}}\mathcal{L}_{f\circ\bar{\sigma}})=0,~~r_{e}(ess_{\mathcal{C}^{k}}\mathcal{L}_{f\circ\bar{\sigma}})=\frac{1}{2^{k}}, ~~\mbox{and}~~r_{e}(ess_{BV}\mathcal{L}_{f\circ\bar{\sigma}})=1/2.
\end{equation}

Hence if $\sigma$ is a uniform distribution on $\mathbb{S}_{n}$, then we are in the setting of Theorem \ref{th:main}.

\begin{theorem}\label{cor3}
Let $ d = 1 / r^{2}$, and $\Upsilon$ be either $\mathcal{A},\mathcal{C}^{k}$ or $BV$.
Then for any $0 < c_0 <  1$, there exists a constant $ c_1 > 0 $ (depending only on $c_0$) such that for all $n \geq d $,
$$
\dP \PAR{ \tau^{\Upsilon}_{f \circ \bar \sigma} \geq  (1  + \veps) \rho }\leq n^{-c_0},
$$
where $$\rho = \frac{1}{\sqrt{r}} \AND \veps =c_1 \frac{\log d }{\sqrt{ \log n}}.$$
\end{theorem}
Together with \cite[Theorem 2]{MR3021362}, we have the following corollary.
\begin{corollary}\label{cor:physical}
For all $n\geq r$ with $gcd(n,r)=1$, then we have
\begin{equation}\label{eq:ineqY}
r_{e}(ess_{\Upsilon}(\mathcal{L}_{f\circ\bar{\sigma}}))\leq\min_{\sigma\in \mathbb S_{n}}\tau^{\Upsilon}_{f\circ \bar{\sigma}}\leq\limsup_{n\to\infty}\mathbb{E}_{n}(\tau^{\Upsilon}_{f\circ\bar{\sigma}})=\frac{1}{\sqrt{r}}<\max_{\sigma\in \mathbb S_{n}}\tau^{\Upsilon}_{f\circ\bar{\sigma}}= \frac{\sin(r\pi/n)}{r\sin(\pi/n)}<1.
\end{equation}
\end{corollary}

Corollary \ref{cor:physical} has an interesting physical interpretation: First of all, the decay of correlation for $f$ itself is always fastest among all the permutations, and it varies on the different regularity choice of observations, e.g. for analytic observations, it is super-exponential with $\tau^{\mathcal{A}}=0$; and for $\mathcal{C}^{k}$ observations, it is exponential with $\tau^{\mathcal{C}^{k}}=\frac{1}{r^{k}}$; while for bounded variation observations, it is exponential with $\tau^{BV}=\frac{1}{r}$ respectively. However, no matter which regular observations are chosen, combining with permutation in shuffling and folding can always decelerate the decay of correlation to arbitrarily slow, providing that the order of the permutation becomes sufficiently large.

On the other hand, regarding for a typical permutation, the average rate can be worse asymptotically at most to $\frac{1}{\sqrt{r}}$, which is independent of the regularity of observations. In other words, if one take an typical interval exchange transformation (not necessarily with the same size of the cell) in practice, then the boundary of interval exchange transformation will be rational, and can be equivalently addressed as a permutation of a very high order. Thus, the mixing rate is becoming slow, but at most to $\frac{1}{\sqrt{r}}$.

\subsection{Strategy of proof of Theorem \ref{th:main}}

The proof of Theorem \ref{th:main} will follow the strategy developed in  \cite{MR3758726,bordenaveCAT} to study the spectral gap of non-backtracking operators of random graphs.  Let us summarize the strategy of proof and its caveats. We will fix an integer $\ell$ of order $\log n  $.  Since \eqref{eq:PerronQ} also holds for $P$, it is immediate to check that
\begin{equation}\label{eq:basicl}
\ABS{\lambda_2}^\ell  \leq \| (P^{\ell}) _{ | \IND^{\perp}} \| := \max_{\langle v , \IND \rangle   = 0} \frac{\| P^\ell v \|_2}{\|v\|_2}.
\end{equation}

Our main result is an upper bound for the operator norm of $P^\ell$ on $\IND^\perp$. By adjusting the constants $c_0,c_1$, Theorem \ref{th:main} is an immediate consequence of \eqref{eq:basicl} and the following result applied to $\ell \sim (c_0/3) \log n / \log d$.

\begin{theorem}\label{th:main2}
For any $0 < c_0 < \delta \leq 1$,  there exists a  constant $c_1 > 0 $ such that,  for any integer $\ell \geq 1$,
$$
\dP \PAR{\| (P^{\ell}) _{ | \IND^{\intercal}} \| \geq e^{ c_1 \sqrt{ \log n}}  \rho^\ell } \leq d^{ \ell+50 \sqrt{ \log n} } n^{- c_0}.
$$
\end{theorem}

To prove Theorem \ref{th:main2}, it would seem natural to introduce the matrix $\underline P =  \underline M Q$ where
\begin{equation}\label{eq:defMbar}
\underline M = M - \frac{1}{n} \cdot \IND \otimes \IND = M - \dE M,
\end{equation}
and
$$
\IND \otimes \IND = \IND \IND^\intercal.
$$
 Indeed, from  \eqref{eq:PerronQ},
$$
 \| (P^{\ell}) _{ | \IND^{\intercal}} \|=  \| (\underline P) ^\ell \|.
$$

A usual route would then be estimating the operator norm $\| (\underline P) ^\ell \|$ thanks to the high trace method. That is, we use for any  real random matrix $B$ and integer $m \geq 1$,
\begin{equation}\label{eq:illC}
\dE \| B \|^{2m}  = \dE \| BB^\intercal \|^m  \leq \dE \tr [\PAR{ BB^\intercal }^m ]
\end{equation}

Our problem requires to use the above inequality with $\ell m \gg \log n$. However, as explained above, due to the potential presence of low dimensional invariant subspaces in $P$,  the event $\lambda_ 2 = 1$ has probability at least $n ^{-c}$ and hence $\dE \| (\underline P)^\ell \| ^{2m}  \geq   n^{-c} $,  which may be much  larger than $\rho (1+ \veps)^{2\ell m}$ for $\veps$ small enough, in the regime $\ell m \gg \log n$.

To circumvent this difficulty, we have to remove beforehand some events.
 We will then use the crucial fact that with high probability the random matrix $M$ is {\it free of $\ell$-tangles} with the matrix $Q$, where a tangle is a path of length $\ell$ which contains at least two cyles in a graph associated to the non-zero entries of $P = MQ$ and $Q$ or meet the subset $\mathcal{E} \subset [n]$ (see Definition \ref{def2} below for a precise definition). On this event, we will have the matrix identity
$$
P^{\ell} = P^{(\ell)},
$$
where $P^{(\ell)}$ is a matrix where the contribution of all tangles will vanish at once (see \eqref{def-A-free} below).  Thanks to basic linear algebra, we will then project the matrix $P^{(\ell)}$ on the orthogonal of the vector $\IND$ and give a deterministic upper bound  of
 $
  \| (P^{\ell}) _{ | \IND^{\intercal}} \|
 $
in terms of the operator norms of new matrices which will be expressed as weighted paths of length at most $\ell$.

In the remainder of the proof, we will use the high trace method to upper bound the operator norms of these new matrices: if $A$ is such matrix, we will use \eqref{eq:illC} for some integer $m$ of order $\sqrt{ \log n}$. By construction, the expression on the right-hand side of \eqref{eq:illC} is then an expected contribution of some weighted paths of lengths $2 m \ell $ of order $\ell \sqrt{\log n}$.

The study of the expected contribution of weighted paths in \eqref{eq:illC} will have a probabilistic and a combinatorial part.  The necessary probabilistic computations on the random permutation are gathered in Section \ref{sec:RP}.  In  Section \ref{sec:PC}, we will use these computations together with combinatorial upper bounds on directed paths to deduce sharp enough bounds on our operator norms. The success of this step will essentially rely on the fact that the contributions of tangles vanish in $P^{(\ell)}$. Finally, in Section  \ref{sec:end}, we gather all ingredients to conclude.

In the remainder of the paper, we let $\mathcal{E}$ be a fixed subset of $[n]$ of cardinality at most $n^{1-\delta}$ which achieves the minimum in \eqref{eq:defAd} for $A = Q$.

\section{Path decomposition}
\label{sec:PD}

In this section, we fix $\sigma \in \mathbb{S}_n$ with permutation matrix $M$ and a positive integer $\ell$.
%
Our aim is to derive a deterministic upper bound on  the norm of $(P^\ell )_{\IND^\perp}$ defined in  \eqref{eq:basicl} (in forthcoming Lemma \ref{le:decompBl}) when $M$ and $Q$ satisfy a property which will be called $\ell$-tangled free. This can be studied by an expansion of paths in the graph.  To this end, we introduce some definition.

\begin{definition} \label{def1}
A {\em path of length $k$} is a sequence $ \gamma = (x_1,  y_1,  x_2 ,  \ldots,  x_{k} ,  y_{k} , x_{k+1} )$, with  $x_t , y_t \in [n]$ and $ Q_{ y_t x_{t+1} }  > 0$.   The set of paths of length $k$ is denoted by $\Gamma^{k}$.  If $x, y \in [n]$, we denote by $\Gamma^{k}_{xy}$   paths in $\Gamma^k$ such that $x_1 = x$, $x_{k+1} = y$.

A {\em subpath} of $\gamma$ is a path of the form $ (x_s, y_s , \ldots,  y_t, x_{t+1} )$ with $1 \leq s \leq t \leq k$, or, if $x_i = x_{j}$ for some $1 \leq i < j \leq n$, a path of the form $ (x_s, y_s , \ldots, x_i, y_j  , \ldots, x_{t+1} )$  with $1 \leq s \leq i <  j \leq t \leq k$.
\end{definition}

 We will use the convention that a product over an empty set is equal to $1$ and the sum over an empty set is $0$. By construction, for integer $k \geq 0$, from \eqref{eq:defP} we find that
\begin{equation}\label{eq:Pk}
(P^{k}) _{xy} = \sum_{\gamma \in \Gamma^{k} _{xy}} \prod_{t=1}^{k} M  _{x_{t} y_{t}} Q _{y_{t} x_{t+1}} ,
\end{equation}
where the sum is over all paths of length $k$ from $x$ to $y$.  Note that, in the above expression for $P^k$, only the summand depends on the permutation $\sigma$.
Observe that $\underline M$ defined in \eqref{eq:defMbar} is the orthogonal projection of $M$ on $\IND^\perp$.
The matrix $(\underline P)^k = (\underline M Q)^k $ can similarly be written as
$$
((\underline P)^{k}) _{xy} = \sum_{\gamma \in \Gamma^k_{xy}} \prod_{t=1}^{k} \underline M  _{x_{t} y_t} Q _{y_t x_{t+1}}.
$$
 As pointed in introduction, the matrix $(\underline P)^k$ is orthogonal projection of $P^k$ on $\IND^\perp$ but it is not suited for our probabilistic analysis.

 We will  now introduce the central definition of {\it tangled paths}. Recall that $\mathcal{E}\subset [n]$ is a fixed set of cardinality at most $n^{1- \delta}$ which achieves the minimum in \eqref{eq:defAd}.

\begin{definition}\label{def2}
Fix  the integer $h :=  \lceil 20 \sqrt{\log n} \rceil$.
\begin{itemize}
\item A {\em coincidence} is a path $(x_1, y_1, \ldots , x_t, y_t , x_{t+1})$ with  $(x_1, \ldots, x_t)$ {\bf pairwise distinct} such that $ (( Q^\intercal Q )^h )_{x_1 x_{t+1}}    > 0$.
\item An {\em $\mathcal{E}$-coincidence} is a path $(x_1, y_1, \ldots , x_t, y_t , x_{t+1})$ with  $(x_1, \ldots, x_t)$ {\bf pairwise distinct} such that $x_1 = x_{t+1}$ is in $\mathcal{E}$.
\item A  path  $\gamma$ is tangle-free if it contains (as subpaths) at most one coincidence, no $\mathcal{E}$-coincidence. It is tangled otherwise. The subsets of tangle-free paths in $\Gamma^k$ and $\Gamma^{k}_{xy}$ will be denoted by $F^k$ and $F^{k} _{xy}$ respectively.
\item The pair $(M,Q)$ is $\ell$-tangle-free if for any $k \in [\ell]$ and $\gamma  = (x_1,  y_1,  x_2 ,  \ldots,  x_{k} ,  y_k , x_{k+1} ) \in \Gamma^k \backslash F^k$, we have
$$
 \prod_{t=1}^{k} M  _{x_{t} y_t} = 0.
$$
\end{itemize}
\end{definition}

Importantly, note that the definition of paths, coincidences and tangles do not depend on $\sigma$, they depend only on the non-zero entries of $Q$. For example, the set $\Gamma^k$ does not depend on the permutation matrix $M$. Observe also that the condition $(( Q^\intercal Q )^h )_{x x'} >0$ is equivalent to the existence of an integer $0 \leq k \leq h$ and sequences  $(x_0,\ldots,x_k)$, $(y_1,\ldots,y_k)$ such that $x_0 = x$, $x_k = x'$, $(x_0,\ldots,x_k)$ pairwise distinct and for any $s \in [k]$, $\min( Q_{y_s x_{ s-1} } , Q_{y_s, x_s} ) > 0$.

\begin{remark}
Note that by our definition, a path following multiple times the same cycle may not tangled.  For example, assume that $x_1, \cdots, x_t$ are points in $[n] \backslash \mathcal{E}$ such that there does not exist an integer $0 \leq s < h $ and $i \ne j$ with $Q^s_{x_i ,x_j} >0$. Then the following path
$$
\gamma= (x_1,  y_1, x_2, y_2, x_3, y_3, x_4, y_4,  x_2,y_2,x_3, y_3,x_4,y_4, x_2, y_2, x_3, y_3, x_4, y_4, x_5)
$$
is tangle-free. Note however that if one of the $x_j$'s in $\mathcal{E}$ then the path is tangled.
\end{remark}

If the pair $(M,Q)$ is $\ell$-tangle-free then by definition, for any $ k \in[\ell]$ and for any $\gamma$ in $\Gamma_{xy}^k \backslash F_{xy}^k$, the summand on the right-hand side of \eqref{eq:Pk} is zero. Therefore,
\begin{equation}\label{eq:BkBk}
P^{k}   = P^{(k)},
\end{equation}
where $P^{(k)}$ is defined by the following formula
\begin{align}\label{def-A-free}
(P^{(k)} ) _{xy}  : = \sum_{\gamma \in F^{k} _{xy}}  \prod_{t=1}^{k} M  _{x_{t} y_t} Q _{y_t x_{t+1}}.
\end{align}
 For $ k \in[ \ell]$,  we define similarly the matrix $\underline P^{(k)}$ by
\begin{equation}\label{eq:defuA}
( \underline P^{(k)}) _{xy} =\sum_{\gamma \in F^{k} _{xy}} \prod_{t=1}^{k} ( \underline M  )_{x_t y_t } Q _{y_t x_{t+1}}.
\end{equation}
Note that it is not necessarily true that even if the pair $(M, Q)$ is $\ell$-tangle-free  that $(\underline P)^\ell = \underline P^{(\ell)}$. Nevertheless, we may still express $P^{(\ell)}v$ in terms of $\underline P^{(\ell)}v$ for all $v \in \IND^\perp$ at the cost of adding an explicit error term. We start with the following telescopic sum decomposition:
\begin{eqnarray}\label{eq:iopl}
(P^{(\ell)}) _{xy} &  =  & (\uP^{(\ell)} )_{xy}  + \sum_{\gamma \in F^{\ell} _{xy}}    \sum_{k = 1}^\ell \prod_{t=1}^{k-1} (\underline M _{x_{t} y_t})  Q_{y_t x_{t+1} } \cdot   \frac {Q_{y_{k} x_{k+1}}}{ n}   \cdot \prod_{ t= k+1}^\ell  M_{x_{t} y_t} Q_{y_t x_{t+1} },
\end{eqnarray}
which is a consequence of the identity,
$$
\prod_{t=1}^\ell a_t = \prod_{t=1}^\ell b_t  + \sum_{k=1}^{\ell}\prod_{t=1}^{k-1} b_t \cdot  ( a_k - b_k)  \cdot \prod_{t = k+1}^{\ell} a_t.
$$
We now rewrite \eqref{eq:iopl} as a sum of matrix products for lower powers of $\uP^{(k)}$ and $P^{(k)}$ up to some remainder terms.  For $ k  \in [\ell]$,  let $T^{\ell,k}$ denote the set of paths $\gamma =  (x_1, y_1, \ldots, y_\ell , x_{\ell+1})$ such that  (i) $\gamma' = (x_1, y_1, \ldots, y_{k-1}, x_{k} )  \in F^{k-1}$, (ii) $\gamma'' = ( x_{k+1}, y_{k+1}, \ldots , y_{\ell}, x_{\ell+1}) \in F^{\ell- k}$, (iii) $\gamma$ is  tangled. We have the following picture:
$$
\gamma =(\gamma', y_{k}, \gamma'') =  (\underbrace{x_1, y_1, \cdots, y_{k-1}, x_k}_{\gamma' \in F^{k-1}}, \, y_k, \, \underbrace{x_{k+1}, y_{k+1}, \cdots, y_\ell, x_{\ell+1}}_{\gamma'' \in F^{\ell-k}}).
$$
Then, if $T^{\ell,k}_{ xy}$ is the subset of $\gamma \in T^{\ell,k}$ such that  $x_1  =x$ and $x_{\ell+1}  =y$,  we set
\begin{equation}\label{eq:defR}
(R^{(\ell)}_{k} )_{xy}   =  \sum_{\gamma \in T^{\ell,k} _{xy}}  \prod_{t=1}^{k-1} (\underline M _{x_{t} y_t})  Q_{y_t x_{t+1} } \cdot   Q_{y_{k} x_{k+1}}   \cdot \prod_{ t= k+1}^\ell  M_{x_{t} y_t} Q_{y_t x_{t+1} }.
\end{equation}

Let us rewrite \eqref{eq:iopl} as
$$
(P^{(\ell)}) _{xy}   =  (\uP^{(\ell)} )_{xy}  +  \frac{1}{n}   \sum_{k = 1}^\ell \underbrace{ \sum_{\gamma \in F^{\ell} _{xy}}  \prod_{t=1}^{k-1} (\underline M _{x_{t} y_t})  Q_{y_t x_{t+1} } \cdot   Q_{y_{k} x_{k+1}}   \cdot \prod_{ t= k+1}^\ell  M_{x_{t} y_t} Q_{y_t x_{t+1} }}_{\text{denoted by $S(k, x, y)$}}.
$$
For fixed $k \in [\ell]$, let us rewrite the summand $S(k, x, y)$. Using the following equality,
$$
F^{\ell}_{xy} \bigsqcup T_{xy}^{\ell, k} =\bigsqcup_{x_k \in [n]} \bigsqcup_{x_{k+1} \in [n]}  \bigsqcup_{y_{k} \in [n] : Q_{y_{k} x_{k+1} >0} } \Big\{(\gamma', y_{k}, \gamma'')\Big| \gamma'\in F_{x x_k}^{k-1}, \gamma'' \in F_{x_{k+1} y}^{\ell-k}\Big\},
$$
and using the definition  \eqref{eq:defR} for $(R^{(\ell)}_{k} )_{xy}$,  we obtain that
\begin{align*}
S(k, x, y) & = \sum_{x_k \in [n]} \sum_{x_{k+1} \in [n]} \sum_{y_{k} \in [n] }  \sum_{\gamma'\in F_{x x_k}^{k-1}}\sum_{\gamma'' \in F_{x_{k+1} y}^{\ell-k}  }   \prod_{t=1}^{k-1} (\underline M _{x_{t} y_t})  Q_{y_t x_{t+1} } \cdot   Q_{y_{k} x_{k+1}}   \cdot \prod_{ t= k+1}^\ell  M_{x_{t} y_t} Q_{y_t x_{t+1} }  - (R^{(\ell)}_{k} )_{xy}
\\
& =   \sum_{x_k \in [n]} \sum_{x_{k+1} \in [n]} \sum_{y_{k} \in [n] }   (\uP^{(k-1)})_{x x_{k}} \cdot  Q_{y_{k} x_{k+1} }  \cdot (P^{(\ell-k)})_{x_{k+1}y}  - (R^{(\ell)}_{k} )_{xy}
\\
& =  \sum_{x_k \in [n]} \sum_{x_{k+1} \in [n]}    (\uP^{(k-1)})_{x x_{k}}   \cdot (\IND \otimes \IND \cdot Q)_{x_k x_{k+1}} \cdot (P^{(\ell-k)})_{x_{k+1}y}  - (R^{(\ell)}_{k} )_{xy}
\\
& = \Big(\uP^{(k-1)} (\IND \otimes \IND  ) P^{(\ell-k)}\Big)_{xy}- (R^{(\ell)}_{k} )_{xy},
\end{align*}
where at the last line we have used that $Q$ is bi-stochastic: $\IND \otimes \IND \cdot  Q = \IND \otimes (  \IND^\intercal Q) = \IND \otimes \IND$.  Therefore,
\begin{eqnarray*}
P^{(\ell)}   &= & \uP^{(\ell)}   +    \frac 1 {n}  \sum_{k = 1} ^{\ell}  \uP^{(k-1)} (  \IND \otimes \IND  )  P^{(\ell - k)}  -   \frac 1 {n}     \sum_{k = 1}^{\ell} R^{(\ell)}_{k},
 \end{eqnarray*}
where we have set $P^{(0)} = \uP^{(0)}  = I$.  Observe that if $(M,Q)$ is $\ell$-tangle-free, then \eqref{eq:BkBk} and $P$ bi-stochastic imply
$$
\IND^\intercal P^{(\ell - k)}  = \IND^\intercal P^{\ell - k} =  \IND^\intercal.
$$
Hence, if $(M,Q)$ is $\ell$-tangle free  and $\langle v , \IND \rangle = 0$, $\|v\|_2 = 1$,  we find
\begin{eqnarray*}
 \| P^{\ell} v \|_2 &   \leq &  \|  \uP^{(\ell)}  \|   +  \frac 1{ n}   \sum_{k = 1}^\ell \| R^{(\ell)}_{k} \|.\label{eq:decompBkx}
\end{eqnarray*}

We mention here that the method used for the proof of the above inequality appeared already in \cite[Section 3]{Bal15}, \cite[Lemma 6]{BCZ} and \cite[Section 3]{bordenaveCAT}.

We arrive at the following lemma.

\begin{lemma}\label{le:decompBl}
Let $\ell \geq 1$ be an integer and $\sigma  \in  \mathbb{S}_n$ with permutation matrix $M$ be such that the pair $(M,Q)$ is $\ell$-tangle-free. Then,
$$
 \| (P^{\ell}) _{ | \IND^{\intercal}} \| \leq  \|  \uP^{(\ell)} \|  +  \frac 1{n}   \sum_{k = 1}^\ell \| R^{(\ell)}_{k} \| .$$
\end{lemma}

\section{Computations on random permutation}
\label{sec:RP}

In this section, we check that if $\sigma$ is uniformly distributed on $\mathbb{S}_n$ then, with high probability the pair $(M,Q)$ is $\ell$-tangle-free provided that $\ell$ is not too large. We will then state a proposition on the expected product of entries of the permutation matrix $\underline M$.   Recall that $h=  \lceil 20 \sqrt{\log n} \rceil$ was defined in Definition \ref{def2}.

\begin{lemma}\label{le:tangle}
There exists $c >0$  such that for any integer $\ell \geq 1$, the pair $(M,Q)$ is $\ell$-tangle free with probability at least $1 -  c \ell  d^{\ell+2h} n^{-\delta}$.
\end{lemma}

\begin{proof}
We may assume without loss of generality that $\ell \leq n/2$ (otherwise the content of the lemma is empty).
Let us say that a path $\gamma =(x_1, y_1 , \ldots, y_k,   x_{k+1})$ {\em occurs}  if for any $t \in [k]$, $
M  _{x_{t} y_t} = 1$ (that is $\sigma(x_t) = y_t$). If the pair $(M,Q)$ is $\ell$-tangled then at least one of the two following paths occurs for some integers with $1 \leq k + k' \leq \ell$, $1 \leq i \leq k+k'$ and $1 \leq k \leq j \leq k+k' +1$:
\begin{enumerate}[]
\item $(I_{k,k',i})$  There exists a path $(x_1, y_1 , \ldots,   x_{k + k' +1})$, where all $x_t$'s are pairwise distinct except possibly $x_1 = x_{k+1}$ and $x_i = x_{k+k'+1}$ such that $(x_1, y_1 , \ldots,   x_{k+1})$  and $(x_{i}, y_{i} , \ldots,  x_{k+ k'+1})$ are distinct coincidences.
\item  $(I'_{k, k', j})$  There exists a path $(x_1, y_1 , \ldots,  x_{k + k' +1})$ where all $x_t$'s are pairwise distinct except possibly  $x_1 = x_{k + k' +1}$ such that  $(x_{k} ,  y_k , \cdots, x_{j})$ is a coincidence and $(x_1, y_1 ,   \ldots  x_{k + k' +1})$ is a coincidence.
\item$(\II_{k})$ There exists a path $(x_1, y_1 , \ldots, y_k,   x_{k+1})$ which is an $\mathcal{E}$-coincidence.
\end{enumerate}

The configuration $I_{k,k',i}$ describes the situation when $\gamma$ has two consecutive coincidences, $I'_{k,k',j}$ accounts for the possibility that one coincidence is contained in another. $\II_k$ describes the possibility of a closed cycle containing an element in $\mathcal{E}$. 

Let us bound the probability of the two different configurations. Recall that if $\{ a_1, \ldots , a_{t} \}$ and $\{ b_1, \ldots , b_t\}$ are two subsets of cardinal $t$ then
\begin{equation}
\label{eq:dpejd}\dP ( \sigma (a_1) = b_1 , \ldots  , \sigma (a_t) = b_t)  =  \frac{1}{ (n)_t}.
\end{equation}
where $(n)_t = n (n-1) \cdots (n-t+1)$.

Let us start with  $I_{k,k',i}$. Then, there are $(n)_{k+k'-1}$ choices for $(x_j)$, $j \notin \{ k+1, k+k' +1\}$, at most $d^{h}$ choices for $x_{k+1}$ and $x_{k+k'+1}$ and $\| Q^\intercal \|_{1 \to 0} ^{k+k'} \leq d^{k + k'}$ choices for the $y_t$'s (since $Q_{y_t x_{ t+1}} >0$ by the definition of a path). We apply \eqref{eq:dpejd} with $t = k + k'$ and $a_{s} = x_s$,  $b_s = y_{s}$, we arrive at
 $$
 \dP ( I_{k,k',i} ) \leq \frac{ d^{k+k'}  d^{2h} (n)_{k + k'-1}}{ (n)_{k+k'}  } \leq 2 \frac{d^{k+k'+2h}}{n},
 $$
(where the last inequality uses $\ell \leq n/2$).

The same argument gives
 $$
 \dP ( I'_{k,k',j} ) \leq \frac{ d^{k+k'}  d^{2h} (n)_{k + k'-1}}{ (n)_{k+k'}  } \leq 2 \frac{d^{k+k'+2h}}{n}.
 $$

Similarly, for $\II_k$ there are at most $|\mathcal{E}|$ choices for $x_1$, $(n)_{k-1}$ choices for $(x_j)$, $j \notin \{1\}$ and $d^{k}$ choices for the $y_t$'s. From \eqref{eq:dpejd}, we get
$$
 \dP ( \II_{k} ) \leq \frac{ d^{k+h} |\mathcal{E}| (n)_{k-1}}{ (n)_{k}  } \leq 2 \frac{d^{k}|\mathcal{E}|}{n} \leq 2 \frac{d^{k}}{n^{\delta}},
$$
(where we have used the assumption that $| \mathcal{E}|\le n^{1-\delta}$).
%
\end{proof}

Let  $\xbf  = (x_1, \ldots , x_k), \ybf = (y_1, \ldots, y_k) \in [n]^k$. We are interested in estimating for $0 \leq k_0 \leq k$,
$$
\dE \prod_{t= 1} ^{k_0}  \underline M_{x_{t} y_{t}} \prod_{t= k_0+1} ^{k}  M_{x_{t} y_{t}}.
$$
To this end, the {\em arcs} of $(\xbf, \ybf)$ is defined as $$A_{\xbf \ybf}= \{ (x_t,y_t) :  t \in [k] \}.$$ The cardinal of $A_{\xbf \ybf}$ is at most $k$.  The {\em multiplicity} of $e \in A_{\xbf \ybf}$ is $m_e = \sum_{t=1}^k \mathbbm{1}( (x_t, y_t) = e)$. An arc $e = (x,y)$  is {\em consistent}, if $\{ t :   (x_t , y_t) = (x,y) \}  = \{ t :  x_t = x \} =  \{ t :  y_t = y \} $.  It is inconsistent otherwise. The following proposition is proved in  \cite[Proposition 27]{bordenaveCAT}.

\begin{proposition}
\label{prop:Epath}
There exists a   constant $c>0$ such that for any $\xbf  = (x_1, \ldots , x_k), \ybf = (y_1, \ldots, y_k) \in [n]^k$ with $2 k \leq \sqrt{n}$ and any $k_0 \leq k$, we have,
$$
\ABS{ \dE \prod_{t= 1} ^{k_0}  \underline M_{x_{t} y_{t}} \prod_{t= k_0+1} ^{k}  M_{x_{t} y_{t}}} \leq c \, 2^{b} \PAR{ \frac{1 }{ n }}^{a} \PAR{ \frac{  3k   }{ \sqrt{n} }}^{a_1},
$$
where $a = | A_{\xbf \ybf} |$,  $b$ is the number of inconsistent arcs of $(\xbf,\ybf)$ and $a_1$ is the number of  $1 \leq t \leq k_0$ such that $(x_t, y_t)$ is consistent and has multiplicity $1$ in $A_{\xbf \ybf}$.
\end{proposition}

\section{High trace method}

\label{sec:PC}
In this section, we use the high trace method  to derive upper bounds on the operator norms of $\uP^{(\ell)}$ and $R_k^{(\ell)}$ defined respectively by \eqref{eq:defuA} and  \eqref{eq:defR}.

\subsection{Operator norm of $\uP^{(\ell)}$}
\label{subsec:Plk}
In this paragraph, we prove the following proposition.

\begin{proposition} \label{prop:normDelta}
Assume $d \leq \exp (  \sqrt{\log n} )$. For any $c_0 > 0$, there exists $c_1 > 0$ (depending on $c_0, \delta$) such that for any integer $1 \leq \ell \leq \log n $, with probability at least $1 - n^{-c_0}$,
$$ \| \uP^{(\ell)} \| \leq e^{ c_1 \sqrt{\log n} }\rho^{ \ell }.$$
\end{proposition}

Recall the number $h$ defined in Definition \ref{def2}. Let $m$ be a positive integer so that
\begin{equation}\label{eq:mh}
6  m < h.
\end{equation}
 With the convention that $x_{2m + 1} = x_1$, we find from \eqref{eq:defuA},
\begin{align*}
& \|\uP^{(\ell)}  \| ^{2 m} = \|\uP^{(\ell)}{\uP^{(\ell)}}^\intercal  \| ^{m}  \leq \tr \BRA{ \PAR{  \uP^{(\ell)}{\uP^{(\ell)}}^\intercal}^{m}  }
\\
 = & \sum_{x_1, \ldots, x_{2m}}\prod_{i=1}^{m}  (\uP^{(\ell)}) _{x_{2i-1} , x_{2 i}}{(\uP^{(\ell)}}) _{x_{2i+1} , x_{2 i}}
 \\
 = & \sum_{x_1, \ldots, x_{2m}}\prod_{i=1}^{m}  \Big[ \hspace{-6pt}\sum_{\substack{\gamma_{2i-1} \\ \in F_{x_{2i-1}, x_{2i}}^\ell}}  \hspace{-6pt} \prod_{t =1}^\ell     (\underline M)_{x_{2i-1,t}  y_{2i-1,t}}  Q _{y_{2i-1,t} x_{2i-1,t+1}}  \Big]  \hspace{-2pt}\cdot  \hspace{-2pt}\Big[\hspace{-6pt}\sum_{\substack{\gamma_{2i} \\ \in F_{x_{2i+1}, x_{2i}}^\ell }} \hspace{-6pt} \prod_{t =1}^\ell   (\underline M)_{x_{2i,t}  y_{2i,t}}  Q _{y_{2i,t} x_{2i,t+1}}  \Big]
 \\
 = & \sum_{x_1, \ldots, x_{2m}}  \hspace{-6pt} \sum_{\substack{\gamma_1, \ldots , \gamma_{2m} \\ \gamma_{2i-1} \in F_{x_{2i-1}, x_{2i}}^\ell,\\ \gamma_{2i} \in F_{x_{2i+1}, x_{2i}}^\ell }} \hspace{-6pt} \prod_{i=1}^{m}  \prod_{t =1}^\ell     (\underline M)_{x_{2i-1,t}  y_{2i-1,t}}  Q _{y_{2i-1,t} x_{2i-1,t+1}}  \prod_{t =1}^\ell   (\underline M)_{x_{2i,t}  y_{2i,t}}  Q _{y_{2i,t} x_{2i,t+1}} ,
\end{align*}
where we used the notation $\gamma_i = (x_{i,1},y_{i,1}, \ldots, y_{i,\ell} , x_{i,\ell+1}) \in F^{\ell}$.

Now, we define $W_{\ell,m}$ as  the set of  $\gamma = ( \gamma_1, \ldots, \gamma_{2m})$ such that $\gamma_i = (x_{i,1}, y_{i,1}, \ldots, y_{i,\ell} , x_{i,\ell+1}) \in F^{\ell}$ and for all $i \in [m]$,
\begin{equation}\label{eq:bcond}
x_{2i,1} = x_{2i+1, 1} \quad \hbox{ and } \quad x_{2i-1,\ell+1} = x_{2i, \ell+1},
\end{equation}
with the convention that $x_{2m+1, 1} = x_{1, 1}$. Using this  notation, we obtain
\begin{align}\label{eq:truAk2}
  \|\uP^{(\ell)}  \| ^{2 m}  \le  \sum_{\gamma \in W_{\ell,m}}   \prod_{i=1}^{2m}  \prod_{t=1}^{\ell}  (\underline M)_{x_{i,t}  y_{i,t}}  Q _{y_{i,t} x_{i,t+1}}.
\end{align}
Our goal is to estimate the expectation of the above expression thanks to Proposition \ref{prop:Epath} and a counting argument which will rely crucially of the fact an element $\gamma \in W_{\ell,m}$ is composed of $2m$ tangle-free paths, $(\gamma_1, \ldots, \gamma_{2m})$.

We will count the elements in $W_{\ell,m}$ in terms of a measure of the size of their support. For $\gamma = (\gamma_1, \gamma_2, \cdots, \gamma_{2m}) \in W_{\ell, m}$, we define $X_\gamma  = \{ x_{i,t} :  i \in [2m],  t\in [\ell] \}$ and $Y_\gamma = \{ y_{i,t} :  i \in [2m],  t\in [\ell] \}$
. We then consider the graph $K_\gamma$ with vertex set $X_\gamma$  and, for any $x,x'$ in $K_\gamma$, $\{x,x' \}$ is an edge of $K_\gamma$ if and only if
$$
(Q^\intercal Q)_{xx'} > 0.
$$
(That is, there exists $y \in [n]$ such that $\min( Q_{yx}, Q_{y x'}) > 0$).
The graph $K_\gamma$ induces an equivalence relation on $X_\gamma$, where each equivalence class is a connected component of $K_\gamma$. We set
\[
\cc(x) := \text{ the equivalence class of $x$}.
\]
(Note that $\cc$ depends implicitly on $X_\gamma$). By definition, for any $x'\in \cc(x)$ with $x' \ne x$,  there exists a sequence $(x_0,  x_1, \cdots, x_k)$ of distinct points in $X_\gamma$ such that
\[
\text{$x_0 = x, \;  x_k = x'$ \; and \; $(Q^\intercal Q)_{x_{t-1} x_t} > 0$ for any $t \in [k]$.}
\]

The {\em arcs} of $\gamma = (\gamma_1, \gamma_2, \cdots, \gamma_{2m}) \in W_{\ell, m}$, denoted by $A_\gamma$, is the  set of distinct pairs $(x_{i,t}, y_{i,t})$.  We define $W_{\ell,m} (s,a,p)$ as the set of $\gamma \in W_{\ell,m}$ with $s = | X_\gamma |$, $a = |A_\gamma|$ and $s - p$ connected components in  $K_\gamma$. Then taking the expectation in \eqref{eq:truAk2}, we may write
\begin{align*}\label{eq:2sum}
\dE  \|\uP^{(\ell)}  \| ^{2 m}  \le    \dE \sum_{\gamma \in W_{\ell,m} }  \prod_{i=1}^{2m}  \prod_{t=1}^{\ell} \underline M_{x_{i,t}y_{i,t}}  Q _{y_{i,t} x_{i,t+1}}  =  \sum_{s, a, p}\sum_{\gamma \in W_{\ell,m}(s, a, p) }  \mu(\gamma)   q(\gamma) .
\end{align*}
where for $\gamma \in W_{\ell,m}$, we have defined
\begin{equation}\label{eq:defmug}
\mu(\gamma) : =   \dE \prod_{i=1}^{2m}  \prod_{t=1}^{\ell} \underline M_{x_{i,t},y_{i,t}} \AND q(\gamma) =  \prod_{i=1}^{2m}  \prod_{t=1}^{\ell}  Q _{y_{i,t} x_{i,t+1}}
\end{equation}

To estimate the above sum, we decompose further  $W_{\ell,m}(s, a, p) $ into equivalence classes as follows. For $\gamma,\gamma' \in W_{\ell,m} (s,a,p)$, let us say $\gamma   \sim \gamma'$ if there exist a pair of permutations $\alpha$ and $\beta$ in $S_n$ such that the image of $K_\gamma$ by $\alpha$ is $K_{\gamma'}$ and for any $(i,t)$, $x'_{i,t} = \alpha(x_{i,t})$, $y'_{i,t} =  \beta (y_{i,t})$ (where $\gamma' =  (\gamma'_1, \gamma'_2, \cdots, \gamma'_{2m})$ with $\gamma'_i = (x'_{i,1},y'_{i,1}, \ldots, y'_{i,\ell} , x'_{i,\ell+1})$). We define $\cW_{\ell,m}(s,a,p)$ as the set of equivalence classes. An element in $\cW_{\ell,m} ( s, a,p)$ is unlabeled in the language of combinatorics.

We notice that $\mu(\gamma) = \mu(\gamma')$ if $\gamma \sim \gamma'$ and we obtain the bound,
\begin{align}\label{eq:2sum}
\dE  \|\uP^{(\ell)}  \| ^{2 m}  \le   \sum_{s, a, p} | \cW (s,a,p) | \max_{ \gamma \in W(s,a,p)}  \PAR{ | \mu(\gamma) | \sum_{{\stackrel{\gamma'  \in W_{\ell,m} (s,a,p) :}{  \gamma' \sim \gamma}}}     q(\gamma') }.
\end{align}

Our first lemma bounds the cardinality of $\cW_{\ell,m}(s, a, p)$.

\begin{lemma}\label{prop:enumpath}
If $g: = a +  p  - s + 1< 0 $ or $2g + 2m >  p$, then $W_{\ell,m} (s,a,p)$ is empty. Otherwise, we have
$$
| \cW _{\ell,m} (s,a,p) | \leq   2^{4m p } (a s^2 \ell  )^{2 m ( g+3)}.
$$
\end{lemma}

We start with an important lemma on the size of the connected components of $K_\gamma$. It is based on the assumption that each $\gamma \in W_{\ell,m}$ is made of $2m$ tangle-free paths and that $m$ is not too large.

\begin{lemma}\label{lem:orbit}
Let $\gamma \in W_{\ell,m}$. Then for any $x \in X_\gamma$, $\cc(x)$ has at most $4m$ elements.
\end{lemma}

\begin{proof}
The proof is by contradiction. Assume that there exist $x \in X_\gamma$ and $k \geq 2$ such that $2 k m +  1 \leq |\cc(x)| \leq 2 (k+1)m$. Then, from the pigeonhole principle, there exists $i \in [2m]$ such that $\gamma_i$ visits at least $k+1$ distinct vertices in  $\cc (x)$. That is, there exist $1 \leq t_1 < \ldots < t_{k+1} \leq \ell$ such that $z_s := x_{i,t_s}$ are distinct vertices in $\cc (x)$.

Let $B(x,r)$ denote the ball of radius $r$ in the graph $K_\gamma$ around $x$. By definition, $B(x,r)$ is contained in the set of $x' \in X_\gamma$ such that $((Q^\intercal Q)^r)_{x x'} > 0$. We now claim that there exists a pair $( s_1,s_2)$ with $1 \leq  s_1 < s_2 \leq k+1$ such that for any $(s,s') \ne (s_1,s_2)$, with $1 \leq  s < s' \leq k+1$, we have  $B(z_s,h/2) \cap B(z_{s'} , h/2) = \emptyset$. Indeed, otherwise, we could find distinct $s_1 < s_2$ and $s_3 < s_4$ such that the distance between $z_{s_{2p - 1}}$ and $z_{s_{2p}}$ is at most $h$ with $p \in \{ 1,2\}$. In particular, $((Q^\intercal Q)^h)_{z_{s_{2p-1}},z_{ s_{2p} }} > 0$ and this contradicts the assumption that $\gamma_i$ is tangle-free.

It follows also that for any $1 \leq s \leq k+1$, $B(z_s,h/2)$ contains at least  $h/2$ vertices. Indeed, since $k \geq 2$, we may consider $s' \ne s$ such that $\{ s,s'\} \ne \{ s_1,s_2 \}$. Then from what precedes, the distance between $z_{s}$ and $z_{s'}$ is at least $h$ (recall that $h$ is even). In particular, the first $h/2 $ vertices on the shortest path from $z_s$ to $z_{ s'}$ are in $B(z_s,h/2)$. We deduce that for any $s$,
$$
\ABS{ B \PAR{z_s , \frac h 2} } \geq  \frac h 2.
$$

So finally, since $B (z_s , h/2) \cap B(z_{s'}, h/ 2)$ is empty for all unordered pairs $\{s,s'\}$ with $s, s' , s_2$, pairwise distinct, we have proved that
$$
\ABS{ \bigcup_{ s=1 } ^{k+1}  B \PAR{z_s , \frac h 2} } \geq \sum_{ s \ne s_2} \ABS{ B \PAR{z_s , \frac h 2} } \geq  \frac {kh} 2.
$$
On the other end, $\bigcup_{ s=1 } ^{k+1}  B (z_s , h/2)$ is contained in $\cc(x)$. Using that $|\cc(x)| \leq 2 (k+1)m$, we deduce that
$$
  \frac{k h} 2  \leq 2(k+1)m.
$$
Hence, since $k \geq 2$,
$$
h \leq 4 m + \frac{4}k m \leq 6 m.
$$
It contradicts \eqref{eq:mh}.
\end{proof}


\begin{proof}[Proof of Lemma \ref{prop:enumpath}]
The proof of Lemma \ref{prop:enumpath} follows very closely \cite[Lemma 17]{MR3758726} and \cite[Lemma 13]{BCZ}.
In order to upper bound $| \cW_{\ell,m} ( s, a,p) | $, we need to find an efficient way  to encode the  paths $\gamma \in \cW_{\ell,m} ( s, a,p) $ (that is, find an injective map from $\cW_{\ell,m}(s,a,p)$ to a larger set whose cardinality is easier to be upper bounded).

If $\gamma \in W_{\ell,m}$, $i \in [2m]$, $t \in [\ell]$, we set $\gamma_{i,t} = (x_{i,t}, y_{i,t}, x_{i, t+1})$. We shall explore the sequence $(\gamma_{i,t})$ in lexicographic order denoted by $\preceq$ (that is $(i,t)\preceq (i+1,t')$ and $(i,t)\preceq(i,t+1)$). We think of the index $(i,t)$ as a time. We define $(i,t)^-$ as the largest index smaller than $(i,t)$ : $(i,t)^- = (i,t-1)$ if $ t \geq 2$, $(i,1)^- = (i-1,\ell)$ if $i \geq 2$ and, by convention, $(1,1)^- = (1,0)$.

We now define a relevant information on $\gamma$ which characterizes its equivalence class. For  $y \in Y_\gamma$, we define $\bar y$ as the order of apparition of $y$ in the sequence $(y_{i,t})_{i \in [2m],t \in [\ell]}$.  Similarly, for $x \in X_\gamma$, $\bar x$ is the order of apparition of $x$ in $(x_{i,t})_{i \in [2m],t \in [\ell]}$ and $\bar \cc(x)$ is the order  of apparition of $\cc(x)$ among the connected components of $K_\gamma$. Finally, if $x \in X_\gamma$, we set $\vec x = (\bar x , s_x )$, where $s_x$ is the set of $\bar x'$ with  $x'\in X_\gamma$ such that $\bar x'< \bar x$ and $(Q^\intercal Q)_{xx'} >0$. For example $\bar x_{1,1} = \bar y_{1,1} = \bar \cc_{\gamma} (x_{1,1}) = 1$ and $\vec x_{1,1} = (1,\emptyset)$. If $x_{1,2} \ne x_{1,1}$ and  $(Q^\intercal Q)_{x_{1,1} x_{1,2} } >0$, we would have $\vec x_{1,2} = ( 2, \{ 1 \} )$. Finally, we set $\bar \gamma_{i,t} = (\vec x_{i,t} , \bar y_{i,t} , \vec x_{ i,t+1})$. By construction, if the sequence $(\bar \gamma_{i,t})_{i \in [2m],t \in [\ell]}$ is known then the equivalence class of $\gamma$ can be determined unambiguously. We thus need to find an encoding of this sequence $ (\bar \gamma_{i,t})_{i \in [2m],t \in [\ell]}$.

To this end, we start  by building a sequence of non-decreasing directed forests which will allow us to find this compact representation of $\gamma \in W_{\ell,m} (s,a,p)$.  We set $V_\gamma = [s-p]$, $V_\gamma$ will be thought as the set of connected components of $K_\gamma$ ordered by the order of their apparition (since $\gamma \in W_{\ell,m} (s,a,p)$, there are $s-p$ such connected components). We consider the colored directed graph $\Gamma = (V_\gamma, E_\gamma)$ on the vertex set $V_\gamma$ defined as follows. For each time $(i,t)$, we put the directed edge $e_{i,t} := (\bar \cc(x_{i,t}), \bar \cc(x_{i,t+1}) )$ in $E_\gamma$ whose {\em color} is defined as the pair $(\bar x_{i,t}, \bar y_{i,t})$ (note that $\Gamma$ may have loop edges of the form $(c,c)$ or multiple edges of the form $(c,c')$ if $c$ is connected to $c'$ by distinct colored edges). By definition, we have $| E_\gamma |= a$. By \eqref{eq:bcond}, the graph $\Gamma$  is weakly connected, that is, after forgetting the direction of the edges of $\Gamma$, it becomes a connected undirected graph. Hence the genus of $\Gamma$ is non-negative :
 \begin{equation}\label{eq:defgen}
0 \leq g = |E_\gamma|- |V_\gamma|  +1 =  a  - (s -p) + 1 =  a -s + p +1.
 \end{equation}
  This already implies the first claim of the lemma.

We define $\Gamma_{i,t}$ as the subgraph of $\Gamma$ spanned by the edges $e_{j,s}$ with $(j,s) \preceq (i,t)$. We have $\Gamma_{2m,\ell} = \Gamma$. We now inductively define a spanning forest of $\Gamma_{i,t}$ as follows. $T_{1,0}$ has no edge and a vertex set $\{1\}$. We say that $(i ,t)$ is a {\em first time} if adding the edge $e_{i,t}$ to $T_{(i,t)^-}$ does not create a (weak) cycle.  Then, if $(i,t)$ is a first time, we add to $T_{(i,t)^-}$ the edge $e_{i,t}$. It gives $T_{i,t}$. If $(i,t)$ is not a first time, we set $T_{i,t} = T_{(i,t)^-}$. By construction, $T_{i,t}$ is a spanning forest of $\Gamma_{i,t}$. We set $T = T_{2m,\ell}$. Due to \eqref{eq:bcond}, we have the following observations.
\begin{enumerate}[-]
\item  If $i$ is odd, $T_{i,t}$ is weakly connected for all $t \in [\ell]$;
\item  If $i$ is even,  $T_{i,t}$ has at most two (weak) connected components for all $t \in [\ell-1]$ and $T_{i,\ell}$ is weakly connected.
\end{enumerate}
In particular, $T = T_{2m, \ell}$ is a spanning tree of $\Gamma$ viewed as an undirected graph.

 For each even $i$, we define the {\em merging time} $(i,t_i)$ as the smallest time $(i,t)$ such that $T_{i,t}$ is weakly connected. Note that the merging time will be a first time if $t_i \geq 2$.

The edges of $\Gamma \backslash T$ will be called {\em excess edges}. The genus $g$ of $\Gamma$ defined by \eqref{eq:defgen} is also the number of excess edges:
$$
| \Gamma \setminus T| =  |E_\gamma|- |V_\gamma|  +1.
$$
We call $(i,t)$ an  {\em important time} if the visited edge $e_{i,t}$ is an excess edge.

By construction, the path $\gamma_i$ can be decomposed by the successive repetition of
 \begin{itemize}
 \item[(1)] a sequence of first times (possibly empty);
 \item[(2)] an important time or the merging time;
 \item [(3)] a path using the colored edges of the forest defined so far (possibly empty).
\end{itemize}

Recall that there is at most one path between two vertices of an oriented forest. Hence, in step (3), it is sufficient to know the starting and ending point to recover the path followed.

We can now build a first encoding of the sequence $ (\bar \gamma_{i,t})_{i \in [2m],t \in [\ell]}$. Assume that the sequence $(\bar \gamma_ {j ,s})_{(j,s) \prec (i,t)}$ is known and that we seen so far $u$ vertices in $X_\gamma$ and $v$ elements in $Y_\gamma$. Then, we observe that if $(i,t)$ is a first time and not the merging time, $\bar \gamma_{i,t}$ is fully determined:
 \begin{enumerate}[-]
\item  if $ t \geq 2$ or $t =1$ and $i$ odd, $\vec x_{i,t} = \vec x_{(i,t)_- +1}$, $\vec x_{i,t+1} = (u+1,\emptyset)$ and $ \bar y_{i,t} = v+1$,
 \item if $t =1$ and $i$ even, $\vec x_{i,1} = (u+1,\emptyset)$, $\vec x_{i,2} = (u+2,\emptyset)$ and $\bar y_{i,1} = v+1$.
 \end{enumerate}
  Indeed, if $t \geq 2$ or $t =1$ and $i$ odd, we have  $\vec x_{i,t} = \vec x_{(i,t)_- +1}$ by \eqref{eq:bcond}. Also, since $(i,t)$ is a first time and not the merging time, $\cc(x_{i,t+1})$ has not been seen before. In particular, $x_{i,t+1}$ has not been seen before and for any $(j,s)\prec (i,t) $, $(Q^\intercal Q)_{x_{j,s} x_{i,t+1}} = 0$. It follows that $ \vec x_{i,t+1} = (u+1,\emptyset)$. Moreover, if we had $y_{i,t} =  y_{j,s}$ for some $(j,s)\prec (i,t)$, then, by definition, $Q_{y_{j,s} x_{j,s+1}} > 0$ and $Q_{y_{j,s} x_{i,t+1}} = Q_{y_{i,t} x_{i,t+1}} >0$. In particular, $(Q^\intercal Q)_{x_{j,s+1} x_{i,t+1}} >0$, this contradicts that $\cc(x_{i,t+1})$ has not been seen before. We deduce that $\bar y_{i,t} = v+1$. The case $t=1$ and $i$ even is similar.

 If $(i,t)$ is an  important time, we mark the time $(i,t)$  by the vector $(\bar y_{i,t}, \bar x_{i,t+1}, \bar x_{i,\tau})$, where $(i,\tau)$ is the next step outside $T_{i,t}$ (by convention, if the path $\gamma_i$ remains on the forest, we set $\tau = \ell+1$). By construction, $(i,\tau)$ is also the next first, important or merging time. Note that $  x_{i,t+1}$ or $ x_{i,\tau}$ could be seen for the first time (then by construction, $x_{i,t+1}$ or $x_{i,\tau}$ would belong to a connected component which has already been seen). If this is the case, we replace $\bar x_{i,t+1}$ or $\bar x_{i,\tau}$ by $\vec x_{i,t+1}$ or $\vec x_{i,\tau}$ and we call this extra mark the {\em connected component mark}. Similarly if $(i,t)$ is the merging time, we mark the time $(i,t)$  by the {\em  merging time mark} $(\bar y_{i,t},  \bar x_{i,t+1}, \bar x_{i,\tau})$, where $(i,\tau)$ is the next step outside $T_{i,t}$. Again, if  $  x_{i,t+1}$ or $ x_{i,\tau}$ are seen for the first time, we replace $\bar x_{i,t+1}$ or $\bar x_{i,\tau}$ by the connected component mark. It gives rise to our first encoding of the sequence $ (\bar \gamma_{i,t})_{i \in [2m],t \in [\ell]}$.

Observe that  $p = \sum_{i=1}^{s-p} (l_i -1)$ where $l_i$ is the size of the $i$-th connected component. Hence $p$ is equal to the number of connected component marks and it is upper bounded by the twice the number of excess edges plus the number of merging times:
$$
p \leq 2(g +m).
$$ It proves the second statement of the lemma.

The issue with this first encoding is that the number of important times may be large. This is where the hypothesis that each path $\gamma_i$ is tangle-free comes into play, more precisely, by Lemma \ref{lem:orbit} and \eqref{eq:mh}, the path $\gamma_i$ can visit at most one distinct cycle of $\Gamma$ (since the diameter of a connected graph is at most its number of vertices).

We are going to partition important times into three categories {\em short cycling}, {\em long cycling} and {\em superfluous} times. For each $i$, consider the smallest time $(i,t_0)$ such that $\cc(x_{i,t_0+1} )\in \{ \cc(x_{i,1}), \ldots, \cc(x_{i,t_0}) \}$.  Let $1 \leq \sigma \leq t_0$ be such that $\cc(x_{i,t_0+1} )= \cc(x_{i,\sigma})$. By assumption, $C_i = (\bar \cc(x_{i,\sigma}), \ldots, \bar \cc(x_{i,t_0+1}))$ will be the unique cycle of $\Gamma$ visited by $\gamma_i$. The last important time $(i,t) \preceq (i,t_0)$ will be called the {\em short cycling} time. We denote by $(i,\hat t)$ the smallest time   $(i,\hat t) \succeq (i,\sigma)$ such that $\bar \cc (x_{i,\hat t+1})$ is not in $C_i$ (by convention $\hat t = \ell+1$ if $\gamma_i$ remains on $C_i$).  If $\hat t > t_0+2$, this means that the cycle $C_i$ has been visited several times from time $(i, t_0+1)$ to time $(i, \hat t)$. We modify  the mark of the short cycling time as  $(\bar y_{i,t}, \bar x_{i,t+1},\sigma, \hat t, \bar x_{i,\tau})$, where $(i, \tau)$, $ \tau \geq \hat t$, is the next step outside $T_{i,t}$ (it is the next first or important time after $(i,\hat t)$, by convention $ \tau = \ell +1$ if the path remain on the tree). Important times $(i,t')$ with $1 \leq t' <  t$ or $\tau \leq t' \leq \ell$ are called long cycling times. The other important times are called superfluous. The key observation  is that for each $ i \in [2 m]$, the number of long cycling times in $\gamma_i$ is bounded by  $g-1$ (since there is at most one cycle, no edge of $\Gamma$  can be seen by $\gamma_i$ twice outside the time interval  between $(i,t+1)$ and $(i,\tau)$, the $-1$ coming from the fact that the short cycling time is an important time).

We now have our second encoding. We can reconstruct the sequence $ (\bar \gamma_{i,t})_{i \in [2m],t \in [\ell]}$ from the positions of the merging times, the long cycling  and the short cycling times and their respective marks. For each $i$, there are at most $1$ short cycling time, $1$ merging time and $g-1$ long cycling times. There are at most $ \ell ^{2m (g+1)}$ ways to position them. By Lemma \ref{lem:orbit}, for any $x$, the number of $x'$ such that $(Q^\intercal Q)_{xx'} > 0$ is at most $4m$.  Hence, there are at most $2^{4m}$ possibilities for a connected component mark. Also, note that $|Y_\gamma| \leq a$ for any $\gamma \in W_{\ell,m}(s,a,p)$. Thus, there are at most  $ a s^2 $ different possible marks for a long cycling time and $a s^2 \ell^2 $ marks for a short cycling time. Finally, for even $i$, there are also at most $a s^2$ possibilities for the merging time mark. We deduce that
\begin{eqnarray*}
| \mathcal{W} _{\ell,m} (s,a,p) |& \leq &  \ell^{2 m (g+1)} ( 2^{4m }  )^{p}  (a s^2 )^{m}   ( a s^2   ) ^{2m (g-1)}(a s^2  \ell^2) ^{2m}. \\
& \leq &  \ell^{2 m (g+3)} 2^{4m p} (a s^2  )^{ 2m (g+1) }.
\end{eqnarray*}
 We find the last statement of the lemma.
\end{proof}

The sum of $q(\gamma)$ for elements in a single equivalence class. Recall the notion of multiplicity defined above Proposition \ref{prop:Epath}, the multiplicity of an arc $(x,y) \in A_{\gamma}$ is the number of times $(i,t)$ such that $(x_{i,t} , y_{i, t}) = (x,y)$.

\begin{lemma}\label{le:sumpath} Assume further that $m \leq \frac{\delta}{ 8} \frac{ \log n }{\log d}$. Then, there exists a constant $c>0$ (depending on $\delta$) such that for any $\gamma \in W_{\ell,m}(s,a,p)$,
\begin{equation*}
\sum_{\stackrel{\gamma'  \in W_{\ell,m} (s,a,p) :}{  \gamma' \sim \gamma}} q(\gamma') \leq   c d ^{2 g + 2 (m -1)+ a_1 + p} n^{s-p}  \rho^{2 \ell m},
\end{equation*}
where $g = a - s + p +1$ and $a_1$ is the number of arcs of $A_\gamma$ with multiplicity one.\end{lemma}

\begin{proof}
The proof relies on a decomposition of the product $q(\gamma)$ over edges in the graph $\Gamma = (V_\gamma,E_\gamma)$ defined in the Lemma \ref{prop:enumpath}. Let $ e= (u,v)$ be an edge of $\Gamma$ with color $(\bar x,\bar y)$ and multiplicity $k = k(e)$. Let us define the out-degree $b = b(e)$ as the number of distinct elements $\bar x_{i,t+1}$ such that $(\bar x_{i,t},\bar y_{i,t}) = (\bar x, \bar y)$ (in words, $b$ is the number of distinct elements in the $v$-th connected component which are visited immediately after a visit of $(\bar x, \bar y)$). Now, the product $q(\gamma)$ can be decomposed as
\begin{equation}\label{eq:qmulti}
q(\gamma) = \prod_{e \in E_\gamma} Q^{k_1}_{y x_1} \cdots Q^{k_{b}}_{y x_{d}}.
\end{equation}
where $e = (u, v)$ is a generic edge as above and $k_1 + \cdots + k_{b} = k$, $k_j \geq 1$ and $x_1, \cdots , x_b$ are in the $v$-th connected component of $\gamma$.

We thus have the upper bound
\begin{equation}\label{eq:qgf}
\sum_{\gamma' :  \gamma' \sim \gamma} q(\gamma') \leq \sum_{\star} \prod_{e \in E_\gamma} \PAR{ \sum_y Q^{k_1}_{y x'_1} \cdots Q^{k_{d}}_{y x'_{d}}},
\end{equation}
where the first sum $\displaystyle{\sum_{\star}}$ is over all possible choices for the elements in $X_{\gamma'}$.

To help the reader, let us first assume that  $ \| Q \|^{(\delta)}_{1 \to \infty} = \| Q \|_{1 \to \infty}$ (for example if $\delta = 1$). Then $\rho  =  \NRMHS{ Q} \vee  \| Q \|_{1 \to \infty}$. If  $e = (u, v)$ is a generic edge as above, then 
\begin{eqnarray}
\sum_y Q^{k_1}_{y x_1} \cdots Q^{k_{b}}_{y x_{b}}   & \leq  & \| Q^\intercal \|_{1\to 0}  \| Q \|_{1 \to \infty}^{k} \leq d \rho^k,\label{eq:d2}
\end{eqnarray}
where we have  used
$$
 \| Q^\intercal \|_{1\to 0} \leq  \| Q^\intercal Q \|_{1\to 0}  = d.$$
Besides, if $b =1 $ and $k \geq 2$, we also have the bound
\begin{eqnarray}
\sum_y Q^{k}_{y x_1}    \leq     \sum_y Q^{2 }_{y x_1} \| Q\|^{k -2}_{ 1 \to \infty} \leq \rho^{k-2}  \sum_y Q^{2 }_{y x_1} \label{eq:d1m2}.
\end{eqnarray}
%
%

We now partition the edges $e = (u,v)$ with color $(\bar x,\bar y)$, multiplicity $ m $ and in-degree $d$ in $E_\gamma$ in three sets, $E_1$ is the set of edges of multiplicity $k=1$. $E_{21}$  is the set of edges such that  $k \geq 2$ and the $v$-th connected component is a singleton. Finally $E_{22}$ is the set of edges such that  $k \geq 2$ and the $v$-th connected component has at least two elements. Note that any edge $e \in E_{1} \cup E_{21}$  has out-degree $b = 1$ and by definition $a_1 = |E_1|$. If $e$ is in $E_1 \cup E_{22}$, we use \eqref{eq:d2},  if $e$ is in $E_{21}$, we use \eqref{eq:d1m2}.  For any $\gamma' \in W_{\ell,m} (s,a,p)$, $\gamma' \sim \gamma$, we arrive at
\begin{equation}\label{eq:bdqg}
q(\gamma') \leq \prod_{e \in  E_{1} \cup E_{22} } (d \rho^k) \prod_{e \in E_{21} }  (\rho^{k-2} \sum_y Q^{2 }_{y x'_1}  ),
\end{equation}
where in the second product, if $e = (u,v) \in E_{21}$, $x'_1 \in X_{\gamma'}$ is the unique  element in the $v$-th connected component of $\gamma'$.

We may now estimate the \eqref{eq:qgf}. There are at most $n^{s-p} d^p$ choices for the different elements in $X_{\gamma'}$. The term $n^{s-p}$ accounts for the possibilities of the first element in each of $s-p$ connected components. The term $d ^p = \| Q^\intercal Q\|_{1 \to 0}^p$ is an  upper bound on the choices for the remaining $p$ elements in the connected components (we add the elements one by one in each connected component in an order which preserves connectivity and we use that for any $x$ there at most $\| Q^\intercal Q\|_{1 \to 0}$ other $x'$ such that $(Q^\intercal Q )_{xx'} >0$). In \eqref{eq:bdqg}, if $e$ is in $E_{21}$, we may sum over all $x'_1 \in [n]$ (the possibilities for the unique vertex in the $v$-th connected component), we get
\begin{eqnarray}
\sum_{\gamma' :  \gamma' \sim \gamma} q(\gamma')& \leq &n^{s-p} d^{p}  \prod_{e \in  E_{1} \cup E_{22} } (d \rho^k) \prod_{e \in E_{21} }  (\rho^{k-2}\NRMHS{ Q}^{2}  ) \nonumber  \\
& = & n^{s-p} d^{p + a_1+|E_{22}|} \rho^ {2 \ell m} \label{eq:sumazaza},
\end{eqnarray}
 where we have used that the sum of the multiplicities is equal to $2\ell m$.

It remains to give an upper bound on $|E_{22}|$. To this end, let $s_k$ (respectively $s_{\geq k}$) be the set of vertices of $\Gamma$ of in-degree $k$ (respectively $\geq k$). We have
$$
s_0 + s_1 + s_{\geq 2} = s-p \AND s_1 + 2 s_{\geq 2}  \leq \sum_k k s_k = a.
$$
Subtracting to the right-hand side, twice the left hand side,
\begin{equation*}
s_{1} \geq 2 (s -p) -  a  - 2 s_0 \geq  a - 2g  - 2m + 2.
\end{equation*}
Indeed, at the last step the bound $s_0 \leq m$ follows from the observation that only a vertex $u \in V_\gamma$ such that $u = \bar \cc (x_{j,1})$ for some $ 1 \leq j \leq 2m$ can be of in-degree $0$. We observe also that $s_1  \leq a_1 + |E_{12}|$ (vertices of in-degree $1$ are in bijection with their unique incoming edge, which cannot be in $E_{22}$). In particular,
\begin{equation}\label{eq:bdE22}
|E_{22}| = a - a_1 - |E_{12}| \leq a - s_1  \leq 2g  + 2m - 2.
\end{equation}
It concludes the proof when $ \| Q\|^{(\delta)}_{1 \to \infty} = \| Q  \|_{1 \to \infty}$.

In the general case, the  bounds \eqref{eq:d2}-\eqref{eq:d1m2}  remain valid except when $x_j$ or $x$ belong to $\mathcal{E}$.  To deal with this case, we first observe the inequality
$$
1 = \PAR{ \sum_{y} Q_{yx} } ^2 \leq \| Q^\intercal \|_{1 \to 0}  \sum_{y} Q^2_{yx} \leq d  \sum_{y} Q^2_{yx}.
$$
Summing over $x$, it implies that
\begin{equation}\label{eq:lbQHS}
 \frac 1 {\sqrt d} \leq  \NRMHS{Q} \leq \rho.
\end{equation}
Hence, in \eqref{eq:d2}-\eqref{eq:d1m2} when $x_j$ or $x$ belong to $\mathcal{E}$, we may use the inequality $Q_{yx} \leq 1 \leq \sqrt{d} \rho$. With the argument leading to \eqref{eq:bdqg}, we obtain for any $\gamma' \in W_{\ell,m} (s,a,p)$, $\gamma' \sim \gamma$,
\begin{equation}\label{eq:bdqg2}
q(\gamma') \leq d^{ u / 2} \prod_{e \in  E_{1} \cup E_{22} } (d \rho^k) \prod_{e \in E_{21} : x'_1 \notin \mathcal E }  (\rho^{k-2} \sum_y Q^{2 }_{y x'_1}  ) \prod_{e \in E_{21} : x'_1 \in \mathcal E } \rho^k,
\end{equation}
where  $u = u_{\gamma'}$ is the  number of times $(i,t)$, $i \in [2m], t \in [\ell]$ such that $x'_{i,t+1} \in \mathcal{E}$ and
Now, for any $\gamma' \in W_{\ell,m} (s,a,p)$ with $\gamma' \sim \gamma$, let $r = r_{\gamma'}$ be the number of connected components which contain at least one element in $\mathcal E$. We claim that the number $u_{\gamma'}$ defined in \eqref{eq:bdqg2} satisfies
$$
u \leq 4 m r.
$$
Indeed, since $\gamma_i$ is tangle-free for each $i \in [2m]$, $\gamma_i$ visits at most once each element in $\mathcal{E}$ (to avoid a $\mathcal{E}$-coincidence) and at most $2$ distinct elements in each connected components (to avoid two or more than two coincidences). Hence, for each $i \in [2m]$, the number of  $t \in [\ell]$ such that $x'_{i,t+1} \in \mathcal{E}$ is at most $2r$. It gives the claimed bound.

We thus deduce from \eqref{eq:bdqg2} that
\begin{equation}\label{eq:bdqg3}
q(\gamma') \leq d^{2 m r } \prod_{e \in  E_{1} \cup E_{22} } (d \rho^k)  \prod_{e \in E_{21} : x'_1 \notin \mathcal E }  (\rho^{k-2} \sum_y Q^{2 }_{y x'_1}  ) \prod_{e \in E_{21} : x'_1 \in \mathcal E } \rho^k,
\end{equation}

Now, in view of \eqref{eq:bdqg3}, we should upper bound the number of $\gamma' \in W_{\ell,m} (s,a,p)$, $\gamma' \sim \gamma$ such that $r_{\gamma'} = r$. A rough upper bound is given by
$$
{ s- p \choose r } n^{s-p -r }  (| \mathcal E| d^{4m} )^r d^p \leq n^{s-p} d^p ( s d^{4m} n^{-\delta} )^r.
$$
Indeed, on the left hand side, the binomial term bounds the number of choices for the connected components which contain at least one element in $\mathcal E$. As pointed above, the term $d^p$ bounds the possibilities for all but the first element in each connected component. Finally the term $| \mathcal E| d^{4m}$ is an upper bound for the number of possibilities  of the first element of a connected element which contains an element in $\mathcal E$ (by Lemma \ref{lem:orbit}, for any such element, say $x_0$, there exists a sequence $(x_0, \ldots, x_{4m})$ such that $x_{4m} \in \mathcal E$ and $(Q^\intercal Q)_{x_{s-1} x_{s}} > 0$ for all $s \in [4m]$).

Hence, from  \eqref{eq:bdqg3}, the argument leading to  \eqref{eq:sumazaza} gives the upper bound
\begin{eqnarray*}
\sum_{\gamma' :  \gamma' \sim \gamma} q(\gamma')& \leq & n^{s-p} d^{p + a_1+|E_{22}|} \rho^ {2 \ell m} \sum_{r=0}^{s-p} ( s d^{6m} n^{-\delta} )^r .
\end{eqnarray*}
We have $s \leq 2 \ell m \leq 10 \lceil \log n \rceil ^{3/2}$ from \eqref{eq:mh}. Hence the assumption $d^{8m} \leq n^{\delta}$ implies that $( s d^{6m} n^{-\delta} ) \leq 1/2$ for all $n$ large enough. It follows  that, for all $n$ large enough, the above geometric series is bounded by $2$ and
\begin{eqnarray*}
\sum_{\gamma' :  \gamma' \sim \gamma} q(\gamma')& \leq & 2 n^{s-p} d^{p + a_1+|E_{22}|} \rho^ {2 \ell m} .
\end{eqnarray*}
From \eqref{eq:bdE22}, it concludes the proof.
\end{proof}

Recall the definition \eqref{eq:defmug} of $\mu(\gamma)$ of the average contribution of $\gamma$ in \eqref{eq:truAk2}.  Our final lemma will use  Proposition \ref{prop:Epath} to estimate this average contribution.

\begin{lemma}\label{le:meanpath}
There is a constant $c > 0$  such that,  if $\gamma \in W_{\ell,m} (s,a,p) $, $g= a - s + p +1$ and  $a_1$ is the number of arcs in $A_\gamma$ which are visited exactly once in $\gamma$, then we have
$$
\ABS{ \mu(\gamma)} \leq  c^{m+g}   n^{-a}   \PAR{\frac{ 6\ell m }{\sqrt n} } ^{  (a_1  - 4 g  - 2 m+2p)_+}.
$$
Moreover,
$a_1 \geq 2 ( a - \ell m)$.
\end{lemma}
\begin{proof}
Let $A_1 \subset A_\gamma$ be the set of $e  = (x,y)$ which are visited exactly once in $\gamma$, that is such that
$$
\sum_{i=1}^{2m} \sum_{t=1}^{\ell} \mathbbm{1} (e =( x_{i,t}  , y_{i,t} ) ) =1.
$$
Let  $A'_1 $ be the subset of $A_1 $ of consistent arcs and let $A_* $ the set of inconsistent arcs (recall the definition above Proposition \ref{prop:Epath}). We have
$$
| A'_1| + | A_{*}| \ge | A_1|.
$$

Set $a'_1 =| A'_1 |$ and $a_{\geq 2} = | A_\gamma \setminus A_1|$. That is, $a_{\geq 2}$ is the number of $e \in A_\gamma$ which are visited at least twice.  We have
$$
a_1 + a_{\geq 2} = a \quad  \hbox{ and } \quad a_1 + 2 a_{\geq 2} \leq 2\ell m.
$$
Therefore,
$$
a_1 \geq 2( a - \ell m).
$$
It gives the second claim. Using the terminology of the proof of Lemma \ref{prop:enumpath}, a new inconsistent arc can appear after leaving the forest constructed so far, at a first visit of an excess edge, or at the merging time ($i$ even) of $\gamma_i$, $i \in [2m]$. Every such step can create $2$ inconsistent arcs.  A step outside the forest constructed so far is preceded by the visit of a new excess edge. Hence,  if  $b = |A_* |$, then
$$
b \leq 4 g  +  2 m
$$
 and
 $$
 a'_1 \geq a_1 - b.
$$

The bound on $b$ can be slightly improved. As already pointed in the proof of Lemma \ref{prop:enumpath},  $p = \sum_{i=1}^{s-p} (l_i -1)$ where $l_i$ is the size of the $i$-th connected component. The first visit to any element in the connected component beyond the first will be a new excess edge but it will not create an inconsistent arc. It follows that $b \leq  4 g + 2m - 2p$ and $a'_1 \geq a_1 - 4g - 2m +2p$. It remains to apply  Proposition \ref{prop:Epath}.
\end{proof}

All ingredients have been gathered to prove Proposition \ref{prop:normDelta}.

\begin{proof}[Proof of Proposition \ref{prop:normDelta}]
We define
\begin{equation}\label{eq:choicem}
m = \left\lceil \frac{\delta}{10} \sqrt{ \log n }\right\rceil.
\end{equation}
From \eqref{eq:2sum} and Markov inequality, it suffices to prove that for some $c >0$,
\begin{equation}\label{eq:boundS}
S =  \sum_{s, a, p} | \cW (s,a,p) | \max_{ \gamma \in W(s,a,p)}  \PAR{ | \mu(\gamma) | \sum_{{\stackrel{\gamma'  \in W_{\ell,m} (s,a,p) :}{  \gamma' \sim \gamma}}}     q(\gamma') }  \leq n e^{ c m^2 }  ,
\end{equation}
where $\ell' =\ell+1 +1/m$ and $\mu(\gamma)$ was defined in \eqref{eq:defmug}.

Let $\gamma \in W_{\ell,m} (s,a,p)$ with $a_1$ arcs of multiplicity one. Set $g  = g(s,a,p) = a - s + p -1$, by Lemma \ref{le:sumpath} and Lemma \ref{le:meanpath},
$$| \mu(\gamma) | \sum_{{\stackrel{\gamma'  \in W_{\ell,m} (s,a,p) :}{  \gamma' \sim \gamma}}}     q(\gamma') \leq    c d  ^{2 g + 2( m-1) + a_1+p} n^{s-p}  \rho^{2 \ell m} c^{m+g}   n^{-a}   \PAR{\frac{ 6\ell m }{\sqrt n} } ^{  (a_1  - 4 g  - 2 m + 2p)_+}.
$$
Since
$
d \geq  1
$, we have $d^{a_1} \leq  d^{4 g  + 2  m - 2p}  d^{(a_1  - 4 g  - 2 m+2p)_+}$. Using $a_1 \geq 2 (a - \ell m )$, we deduce the following upper bound, for some new constant $c >1$, $$| \mu(\gamma) | \sum_{{\stackrel{\gamma'  \in W_{\ell,m} (s,a,p) :}{  \gamma' \sim \gamma}}}     q(\gamma') \leq  \PAR{ c  d  }^{6 g + 4 m}  n^{-g+1}  \rho^{2 \ell m}     \PAR{\frac{ (6 d  \ell m )^{2} }{n} } ^{  (a - (\ell +1)m  - 2 g+p)_+}.
$$
For ease of notation, we set
$$
\veps = \frac{(6 d \ell m)^2 } { n} = o(1).
$$
where we have used that $ d \leq \exp ( \sqrt {\log n })$ and $\ell m = O  (\log n)^{3/2}$.
Now by Lemma \ref{prop:enumpath}, since $a \leq 2 \ell m$, $s \leq 2 \ell m + 1 \leq 3 \ell m$, for some new constant $c >1$ changing from line to line, we arrive at
\begin{eqnarray*}
S & \leq &  n  \rho^{2 \ell m}     \sum_{ s , a, p :  g (s,a,p) \geq 0, p \leq2 g (s,a,p) +2 m} 2^{4m p } (a s^2 \ell  )^{2 m ( g+3)}   \PAR{ c  d  }^{6 g + 4m}  n^{-g}  \veps ^{  (s - \ell' m  - g)_+}   \\
 & \leq &  n  (c \ell m)^{24m} (cd)^{4m}  \rho^{2 \ell m}  \sum_{s , g , p : g \geq 0, p \leq 2g + 2m}  2^{4m p } (c \ell  m  )^{8 m g } d^{6 g}  n^{-g}    \veps ^{  (s - \ell' m  - g)_+} ,
 \end{eqnarray*}
 where at the last line, we have performed the change of variable $a \to  g = a +p -s +1$. Then, we may sum over $p$, using $(\log n)^c = e^{ o (m)}$ and $ d \leq e^{10 m/\delta}$, we get for some new constant $c >0$,
\begin{eqnarray*}
S & \leq &  n e^{c m^2} \rho^{2 \ell m}   \sum_{s, g \geq 0}   \PAR{\frac{ L  }{n} }^{g}     \veps ^{  (s  - \ell'  m  -  g)_+},
 \end{eqnarray*}
where we have set $L = (c \ell  m)^{8m} d^6$. We decompose the above sum as  follows
  \begin{eqnarray*}
S  & \leq &  S_1 + S_2 + S_3, \label{eq:Sboundddd}
 \end{eqnarray*}
where $S_1$ is the sum over $\{ 1 \leq s \leq \ell 'm, g \geq 0 \}$, $S_2$ over $\{ \ell' m <  s , 0 \leq g  \leq  s  -  \ell 'm \}$, and $S_3$ over $\{ \ell' m  <  s ,  g >  s - \ell' m  \}$. We start with the first term :
\begin{eqnarray*}
S_1 & =  &   n e^{c m^2} \rho^{2 \ell m}   \sum_{s=1}^{ \ell' m} \sum_{g = 0} ^{ \infty } \PAR{\frac{ L  }{n} }^{g}.
\end{eqnarray*}
For our choice of $m$ in \eqref{eq:choicem}, for some $c>0$ and $n$ large enough,
$$
\frac{   L  }{ n } =  \frac{e^{ c  (\log \log n) \sqrt{ \log n}  } }{n}  \leq \frac 1 2.
$$
In particular,  for $n$ large enough, the above geometric series converges :
$$
S_1 \leq 2 n e^{c m^2} \rho^{2 \ell m}    \sum_{s=1}^{ \ell' m} \leq n e^{c' m^2} \rho^{2 \ell m} .
$$
Adjusting the value of $c'$,  the right-hand side of \eqref{eq:boundS} is an upper bound for $S_1$. Similarly, since $ L   / (\veps n) \geq 2$, we find
\begin{eqnarray*}
S_2 &\leq &  n e^{c m^2} \rho^{2 \ell m}   \sum_{s=\ell' m +1}^{\infty}  \veps^{  s -  \ell' m} \sum_{g = 0} ^{s  -  \ell 'm } \PAR{\frac{ L  }{\veps n} }^{g} \\
& \leq & 2 n e^{c m^2} \rho^{2 \ell m}    \sum_{s=\ell' m +1}^{\infty}  \veps^{  s -  \ell' m}     \PAR{\frac{ L  }{\veps n} }^{s  -  \ell 'm }  \\
& =  & 2 n e^{c m^2} \rho^{2 \ell m}  \sum_{k =1}^{\infty}  \PAR{\frac{ L  }{ n} }^{ k }.
\end{eqnarray*}
Again, for $n$ large enough, the geometric series are convergent and the right-hand side of \eqref{eq:boundS} is an upper bound for $S_2$. Finally,  for $n$ large enough,
\begin{eqnarray*}
S_3 &\leq & n e^{c m^2} \rho^{2 \ell m}   \sum_{s=\ell' m +1}^{\infty}  \sum_{g = s  -  \ell 'm +1 }^\infty  \PAR{\frac{ L  }{n} }^{g} \\
& \leq & n e^{c m^2}  \rho^{2 \ell m}  \sum_{s=\ell' m +1}^{\infty}   2 \PAR{\frac{ L }{n} }^{s - \ell' m +1} \\
& = & 2 n e^{c m^2}  \rho^{2 \ell m}   \sum_{k = 0}^{ \infty}\PAR{\frac{ L  }{n} }^{k}
\end{eqnarray*}
For $n$ large enough, the right-hand side of \eqref{eq:boundS} is an upper bound for $S_3$. It concludes the proof.
\end{proof}

\subsection{Operator norm of $R^{(\ell)}_{k}$}

\label{subsec:Rlk}
We now adapt the above subsection for the treatment of $R^{(\ell)}_{k}$. A rougher bound will suffice for our purposes.

\begin{proposition} \label{prop:normR}
Assume $d \leq \exp ( \sqrt{\log n})$. For any $c_0 > 0$, there exists $c_1 > 0$ (depending on $c_0$) such that with probability at least $1 - n^{-c_0}$, for all integers $1 \leq  k \leq \ell \leq  \log n$,
$$ \| R^{(\ell)}_{k} \| \leq e^{ c_1 \sqrt{\log n}}.$$
\end{proposition}

To help the reader, we use the same notation than in the Subsection \ref{subsec:Plk}, we  add a prime exponent to our objects when the definition differs from the corresponding definition in Subsection \ref{subsec:Plk}.

We fix for some postive integer $m$ such that
\begin{equation}\label{eq:mhR}
12  m < h.
\end{equation}
We use the inequality
$$
  \|R^{(\ell)}_k \| ^{2 m} \leq \tr \BRA{\PAR{ R^{(\ell)}_k {R^{(\ell)}_k}^\intercal } }.
$$
We may expand the trace. To this end, we define $ W'_{\ell,m}$ as  the set of  $\gamma = ( \gamma_1, \ldots, \gamma_{2m})$ such that $\gamma_i = (x_{i,1}, y_{i,1}, \ldots, y_{i,\ell} , x_{i,\ell+1}) \in T^{\ell,k}$ and such that for all $i \in [m]$, the boundary condition \eqref{eq:bcond} holds. Using this  notation, the computation leading to \eqref{eq:truAk2} gives
\begin{align}\label{eq:truAk2R}
  \|R^{(\ell)}_k \| ^{2 m} \le  \sum_{\gamma \in  W'_{\ell,m}}   \prod_{i=1}^{2m}  \prod_{t=1}^{k-1} (\underline M _{x_{i,t} y_{i,t}})  Q_{y_{i,t} x_{i,t+1} } \cdot   Q_{y_{i,k} x_{i,k+1}}   \cdot \prod_{ t= k+1}^\ell  M_{x_{i,t} y_{i,t}} Q_{y_{i,t} x_{i,t+1} }.
\end{align}
We set
$$
\gamma'_i =  (x_{i,1}, y_{i,1}, \ldots, y_{i,k-1} , x_{i,k}) \AND \gamma''_i = (x_{i,k+1}, y_{i,k+1}, \ldots, y_{i,\ell} , x_{i,\ell+1})
$$
By construction $\gamma'_i$ and $\gamma''_i$ are tangled-free paths.

As in Subsection \ref{subsec:Plk}, for $\gamma = (\gamma_1, \gamma_2, \cdots, \gamma_{2m}) \in W'_{\ell, m}$, we define $X_\gamma  = \{ x_{i,t} :  i \in [2m],  t\in [\ell] \}$ and $Y_\gamma = \{ y_{i,t} :  i \in [2m],  t\in [\ell] \}$. We  consider the same graph $K_\gamma$ with vertex set $X_\gamma$  and, for any $x,x'$ in $K_\gamma$, $\{x,x' \}$ is an edge of $K_\gamma$ if and only if $
(Q^\intercal Q)_{xx'} > 0.
$
We denote by $\cc(x)$ the connected component of $x \in X_\gamma$ in $K_\gamma$.  The {\em arcs} of $\gamma = (\gamma_1, \gamma_2, \cdots, \gamma_{2m}) \in W'_{\ell, m}$, denoted by $A'_\gamma$, is the  set of distinct pairs $(x_{i,t}, y_{i,t})$ with $t \ne k$.  We define $W'_{\ell,m} (s,a,p)$ as the set of $\gamma \in W_{\ell,m}$ with $s = | X_\gamma |$, $a = |A'_\gamma|$ and $s - p$ connected components in  $K_\gamma$. We take the expectation in  \eqref{eq:truAk2R} and write \begin{align*}\label{eq:2sum}
\dE  \|R^{(\ell)}_k  \| ^{2 m}  \le    \sum_{s, a, p}\sum_{\gamma \in W'_{\ell,m}(s, a, p) }  \mu'(\gamma)   q(\gamma) .
\end{align*}
where for $\gamma \in W'_{\ell,m}$, we have defined
\begin{equation}\label{eq:defmugR}
\mu'(\gamma) : =   \dE \prod_{i=1}^{2m}  \prod_{t=1}^{k-1} \underline M_{x_{i,t},y_{i,t}} \prod_{ t= k+1}^\ell  M_{x_{i,t}} \AND q(\gamma) =  \prod_{i=1}^{2m}  \prod_{t=1}^{\ell}  Q _{y_{i,t} x_{i,t+1}}
\end{equation}

We decompose further  $W'_{\ell,m}(s, a, p) $ into equivalence classes as follows. For $\gamma,\gamma' \in W'_{\ell,m} (s,a,p)$, let us say $\gamma   \sim \gamma'$ if there exist a pair of permutations $\alpha$ and $\beta$ in $S_n$ such that the image of $K_\gamma$ by $\alpha$ is $K_{\gamma'}$ and for any $(i,t)$, $x'_{i,t} = \alpha(x_{i,t})$, $y'_{i,t} =  \beta (y_{i,t})$ (where $\gamma' =  (\gamma'_1, \gamma'_2, \cdots, \gamma'_{2m})$ with $\gamma'_i = (x'_{i,1},y'_{i,1}, \ldots, y'_{i,\ell} , x'_{i,\ell+1})$). We define $\cW'_{\ell,m}(s,a,p)$ as the set of equivalence classes. Since $\mu(\gamma) = \mu(\gamma')$ if $\gamma \sim \gamma'$, we obtain the bound,
\begin{align}\label{eq:2sumR}
\dE  \|R^{(\ell)}_k  \| ^{2 m}  \le   \sum_{s, a, p} | \cW' (s,a,p) | \max_{ \gamma \in W'(s,a,p)}  \PAR{ | \mu'(\gamma) | \sum_{{\stackrel{\gamma'  \in W_{\ell,m} (s,a,p) :}{  \gamma' \sim \gamma}}}     q(\gamma') }.
\end{align}

We start by bouding the the cardinality of $\cW'_{\ell,m}(s, a, p)$.

\begin{lemma}\label{prop:enumpathR}
If $g': = a +  p  - s < 0 $ or $2 g' + 10 m >  p$, then $W'_{\ell,m} (s,a,p)$ is empty. Otherwise, we have
$$
| \cW' _{\ell,m} (s,a,p) | \leq   2^{4m p } \PAR{(a +2m)^2s^2 \ell  }^{4 m ( g'+4)}.
$$
\end{lemma}

We have the following analog of Lemma \ref{lem:orbit}.

\begin{lemma}\label{lem:orbitR}
Let $\gamma \in  W'_{\ell,m}$. Then for any $x \in X_\gamma$, $\cc(x)$ has at most $8m$ elements.
\end{lemma}
\begin{proof}
We repeat the proof of Lemma \ref{lem:orbit}, we use this time that $\gamma$ is composed of $4m$ tangle-free paths: $\gamma'_i,\gamma''_i$, for $i \in [2m]$.  By contradiction, we assume that there exist $x \in X_\gamma$ and $k \geq 2$ such that $4 k m +  1 \leq |\cc(x)| \leq 4 (k+1)m$. Then, from the pigeonhole principle, there exists $i \in [2m]$ and $\veps \in \{ ', ''\}$ such that $\gamma^{\veps}_i$ visits at least  $k+1$ distinct vertices in  $\cc (x)$. We then repeat verbatim the proof of Lemma \ref{lem:orbit} and use \eqref{eq:mhR}.
\end{proof}

\begin{proof}[Proof of Lemma \ref{prop:enumpathR}]
We repeat the proof of Lemma \ref{prop:enumpath}.
If $\gamma \in W'_{\ell,m}$, $i \in [2m]$, $t \in [\ell]$, we set $\gamma_{i,t} = (x_{i,t}, y_{i,t}, x_{i, t+1})$. We shall explore the sequence $(\gamma_{i,t})$ in lexicographic order denoted by $\preceq$ (that is $(i,t)\preceq (i+1,t')$ and $(i,t)\preceq(i,t+1)$). We think of the index $(i,t)$ as a time. We define $(i,t)^-$ as the largest index smaller than $(i,t)$ and, by convention, $(1,1)^- = (1,0)$.

As in Lemma \ref{prop:enumpath}, for  $y \in Y_\gamma$, we define $\bar y$ as the order of apparition of $y$ in the sequence $(y_{i,t})_{i \in [2m],t \in [\ell]}$.  Similarly, for $x \in X_\gamma$, $\bar x$ is the order of apparition of $x$ in $(x_{i,t})_{i \in [2m],t \in [\ell]}$ and $\bar \cc(x)$ is the order  of apparition of $\cc(x)$ among the connected components of $K_\gamma$. Finally, if $x \in X_\gamma$, we set $\vec x = (\bar x , s_x )$, where $s_x$ is the set of $\bar x'$ with  $x'\in X_\gamma$ such that $\bar x'< \bar x$ and $(Q^\intercal Q)_{xx'} >0$.  Finally, we set $\bar \gamma_{i,t} = (\vec x_{i,t} , \bar y_{i,t} , \vec x_{ i,t+1})$. By construction, if the sequence $(\bar \gamma_{i,t})_{i \in [2m],t \in [\ell]}$ is known then the equivalence class of $\gamma$ can be determined unambiguously. We thus need to find an encoding of this sequence $ (\bar \gamma_{i,t})_{i \in [2m],t \in [\ell]}$.

We set $V_\gamma = [s-p]$ and consider the colored directed graph $\Gamma' = (V_\gamma, E'_\gamma)$ on the vertex set $V_\gamma$ defined as follows. For each time $(i,t)$, with $t \ne k$, we put the directed edge $e_{i,t} := (\bar \cc(x_{i,t}), \bar \cc(x_{i,t+1}) )$ in $E'_\gamma$ whose {\em color} is defined as the pair $(\bar x_{i,t}, \bar y_{i,t})$. By definition, we have $| E' _\gamma |= a$. Let $\bar \Gamma'$ be the associated undirected graph (that is the undirected graph obtained by forgetting the direction of the edges of $\Gamma'$).  We observe that each connected component of $\bar \Gamma'$ contains at least a cycle. Indeed, by assumption $\gamma_i$ is tangled while $\gamma'_i$ and $\gamma''_i$ is tangle-free. Hence if the image of the paths of $\gamma'_i$ and $\gamma'_{ii}$ on $\bar \Gamma'$ do not intersect then each one contains a distinct cycle. Otherwise, the images of the paths intersect, then they are in the same connected component of $\bar \Gamma'$ and their union has at least two distinct cycles. Hence the number of edges of $\Gamma'$ is at least the number of vertices:
 \begin{equation*}\label{eq:defgenR}
0 \leq g' = |E_\gamma|- |V_\gamma|  =   a -s + p.
 \end{equation*}
 This is the first claim of the lemma.

We define $\Gamma'_{i,t}$ as the subgraph of $\Gamma'$ spanned by the edges $e_{j,s}$ with $(j,s) \preceq (i,t)$. We have $\Gamma'_{2m,\ell} = \Gamma'$. As in Lemma \ref{prop:enumpath}, we now inductively define a spanning forest $T_{i,t}$ of $\Gamma'_{i,t}$ as follows. $T_{1,0}$ has no edge and a vertex set $\{1\}$. We say that $(i ,t)$ is a {\em first time} if adding the edge $e_{i,t}$ to $T_{(i,t)^-}$ does not create a (weak) cycle.  Then, if $(i,t)$ is a first time, we add to $T_{(i,t)^-}$ the edge $e_{i,t}$. It gives $T_{i,t}$. If $(i,t)$ is not a first time, we set $T_{i,t} = T_{(i,t)^-}$. We set $T = T_{2m,\ell}$.

 For each even $i$, we define the {\em first merging time} $(i,t'_i)$ as the smallest time $(i,t)$ with $1 \leq t \leq k-1$ such that $T_{i,t}$ and  $T_{(i,1)^-}$ have the same number of connected components. If this time does not exist, we set $t'_i  = k$. Similarly, for each $i$, the {\em second merging time} $(i,t''_i)$ is the smallest time $(i,t)$ with $k \leq t \leq \ell$ such that $T_{i,t}$ and  $T_{(i,k)^-}$ have the same number of connected components. If this time does not exist, we set $t''_i  = \ell+1$.  If $i$ is even then  by \eqref{eq:bcond}, we have $t''_i \leq  \ell$.

 Note that the merging time will be a first time if $t_i \geq 2$.

The edges of $\Gamma' \backslash T$ will be called {\em excess edges}. We call $(i,t)$ an  {\em important time} if the visited edge $e_{i,t}$ is an excess edge. The total number of excess edges is $|E_\gamma|- |V_\gamma|  + N_\gamma= g' + N_\gamma $ where $1 \leq N_\gamma \leq 2m$ is the number of connected components of $\bar \Gamma'$. However, since each connected component has at least a cycle, in each connected component of $T$, there are at most
$g'+1$ excess edges.

By construction, the path $\gamma_i'$ or $\gamma''_i$ can be decomposed by the successive repetition of
 \begin{itemize}
 \item[(1)] a sequence of first times (possibly empty);
 \item[(2)] an important time or the merging time;
 \item [(3)] a path using the colored edges of the forest defined so far (possibly empty).
\end{itemize}

We build a first encoding of the sequence $ (\bar \gamma_{i,t})_{i \in [2m],t \in [\ell]}$ as follows. If $(i,t)$ is an  important time, we mark the time $(i,t)$  by the vector $(\bar y_{i,t}, \bar x_{i,t+1}, \bar x_{i,\tau})$, where $(i,\tau)$ is the next step outside $T_{i,t}$ (by convention, if the path $\gamma_i$ remains on the forest, we set $\tau = \ell+1$). By construction, $(i,\tau)$ is also the next first, important or merging time. Note that $  x_{i,t+1}$ or $ x_{i,\tau}$ could be seen for the first time (then by construction, $x_{i,t+1}$ or $x_{i,\tau}$ would belong to a connected component which has already been seen). If this is the case, we replace $\bar x_{i,t+1}$ or $\bar x_{i,\tau}$ by $\vec x_{i,t+1}$ or $\vec x_{i,\tau}$ and we call this extra mark the {\em connected component mark}. Similarly if $(i,t)$ is a first merging time, we mark the time $(i,t)$  by the {\em  first merging time mark} $(\bar y_{i,t},  \bar x_{i,t+1}, \bar x_{i,\tau})$, where $(i,\tau)$ is the next step outside $T_{i,t}$. Similarly, the {\em second merging time mark} is $(\bar y_{i,k}, \bar y_{i,t},  \bar x_{i,t+1}, \bar x_{i,\tau})$. Again, if  $  x_{i,t+1}$ or $ x_{i,\tau}$ are seen for the first time, we replace $\bar x_{i,t+1}$ or $\bar x_{i,\tau}$ by the connected component mark. Arguing as in the proof of Lemma \ref{prop:enumpath}, it gives a first encoding of the sequence $ (\bar \gamma_{i,t})_{i \in [2m],t \in [\ell]}$.

Observe that  $p = \sum_{i=1}^{s-p} (l_i -1)$ where $l_i$ is the size of the $i$-th connected component of $K_\gamma$. Hence $p$ is equal to the number of connected component marks and it is upper bounded by twice the number of excess edges plus the number of merging times:
$$
p \leq 2\PAR{ g'  + N_\gamma  +3m} \leq 2 g' + 10m.
$$ It proves the second statement of the lemma.

 Arguing as in the proof of Lemma \ref{prop:enumpath}, to improve on the first encoding we use the hypothesis that each path $\gamma'_i$ or $\gamma''_i$ is tangle-free. We partition important times into three categories {\em short cycling}, {\em long cycling} and {\em superfluous} times. For each $i$ and $\veps \in \{' , ''\}$, consider the smallest time $(i,t_0)$ such that $\cc(x_{i,t_0+1} )\in \{ \cc(x_{i,1}), \ldots, \cc(x_{i,t_0}) \}$.  Let $1 \leq \sigma \leq t_0$ be such that $\cc(x_{i,t_0+1} )= \cc(x_{i,\sigma})$. By assumption, $C_i = (\bar \cc(x_{i,\sigma}), \ldots, \bar \cc(x_{i,t_0+1}))$ will be the unique cycle of $\Gamma'$ visited by $\gamma^\veps_i$. The last important time $(i,t) \preceq (i,t_0)$ will be called the {\em short cycling} time. We denote by $(i,\hat t)$ the smallest time   $(i,\hat t) \succeq (i,\sigma)$ such that $\bar \cc (x_{i,\hat t+1})$ is not in $C_i$ (by convention $\hat t = \ell+1$ if $\gamma^\veps_i$ remains on $C_i$). We modify  the mark of the short cycling time as  $(\bar y_{i,t}, \bar x_{i,t+1},\sigma, \hat t, \bar x_{i,\tau})$, where $(i, \tau)$, $ \tau \geq \hat t$, is the next step outside $T_{i,t}$ (it is the next first or important time after $(i,\hat t)$, by convention $ \tau = \ell +1$ if the path remain on the tree). Important times $(i,t')$ with $1 \leq t' <  t$ or $\tau \leq t' \leq \ell$ are called long cycling times. The other important times are called superfluous. As argued in the proof of Lemma \ref{prop:enumpath}, for each $ i \in [2 m]$ and $\veps \in \{' , ''\}$, the number of long cycling times in $\gamma^\veps_i$ is bounded by  $g'$ (recall that there are at most $g'+1$ excess edges in the connected component of $\gamma_i^\veps$).

We now have our second encoding. We can reconstruct the sequence $ (\bar \gamma_{i,t})_{i \in [2m],t \in [\ell]}$ from the positions of the merging times, the long cycling  and the short cycling times and their respective marks. For each $i$ and $\veps \in \{' , '' \}$, there are at most $1$ short cycling time, $1$ merging times and $g'$ long cycling times. There are at most $ \ell ^{4m (g'+2)}$ ways to position them. Note that $|Y_\gamma| \leq a+2m  = a'$, the term $2m$ coming from the elements $y_{i,k}$, $i \in [2m]$. Hence, as argued in the proof of Lemma \ref{prop:enumpath},  there are at most $2^{4m}$ possibilities for a connected component mark, at most $ a' s^2 $ different possible marks for a long cycling time, $a' s^2 \ell^2 $ marks for a short cycling time, at most $a' s^2$ marks for the first merging time mark and ${a'}^2 s^2$ for the second merging time. We deduce that
\begin{eqnarray*}
| \mathcal{W}' _{\ell,m} (s,a,p) |& \leq &  \ell^{4 m (g'+2)} ( 2^{4m }  )^{p}  (a' s^2 )^{m} ({a' }^2s^2 )^{2m}  ( {a'} s^2   ) ^{4m g' }(a' s^2  \ell^2) ^{4m}. \\
& \leq &  \ell^{4 m (g'+4)} 2^{4m p} ({a' }^2 s^2  )^{ 4m (g'+1) }.
\end{eqnarray*}
It concludes the proof.
\end{proof}

\begin{lemma}\label{le:sumpathR} For any $\gamma \in W'_{\ell,m}(s,a,p)$,
\begin{equation*}
\sum_{\stackrel{\gamma'  \in W'_{\ell,m} (s,a,p) :}{  \gamma' \sim \gamma}} q(\gamma') \leq  d^p n^{s-p}.
\end{equation*}
\end{lemma}

\begin{proof}
The proof follows easily from the proof of Lemma \ref{le:sumpath}. Let $\Gamma' = (V_\gamma, E'_\gamma)$ be the graph defined in Proposition \ref{prop:enumpathR}. Arguing as in \eqref{eq:qgf}, we have an upper bound of the form
\begin{equation*}
\sum_{\gamma' :  \gamma' \sim \gamma} q(\gamma') \leq \sum_{\star} \prod_{e \in E_\gamma} \PAR{ \sum_y Q^{k_1}_{y x'_1} \cdots Q^{k_{b}}_{y x'_{b}}},
\end{equation*}
where the first sum $\displaystyle{\sum_{\star}}$ is over all possible choices for the distinct elements in $X_{\gamma'}$, and the positive integers $k_j$ and the elements $x'_j \in X_{\gamma'}$ are determined by the edge $e$. Since $k_j \geq 1$ and $ \sum_y Q_{yx} = 1$, we have
$$
\sum_y Q^{k_1}_{y x'_1} \cdots Q^{k_{d}}_{y x'_{d}} \leq 1.
$$
It follows that
$
\sum_{\gamma' :  \gamma' \sim \gamma} q(\gamma') $ is upper bounded by number of possible choices for  $X_{\gamma'}$. The latter is bounded by $ d^p n^{s-p}$ as explained in the proof of Lemma \ref{le:sumpath}.
\end{proof}

We finally estimate $\mu'(\gamma)$.\begin{lemma}\label{le:meanpathR}
There is a constant $c > 0$  such that,  if $\gamma \in W_{\ell,m} (s,a,p) $, $g= a - s + p $ and  $a_1$ is the number of arcs in $A_\gamma$ which are visited exactly once in $\gamma$, then we have
$$
\ABS{ \mu'(\gamma)} \leq  c^{m+g'}   n^{-a}.
$$
\end{lemma}
\begin{proof}
Let $A_* $ be the set of inconsistent arcs of $A'_\gamma$ (as defined above Proposition \ref{prop:Epath}). Using the terminology of the proof of Proposition \ref{prop:enumpathR} and as argued in Lemma \ref{le:meanpathR}, $|A_*|$ is upper bounded by four times  the number of excess edges plus twice the number of merging times. There are at most $g' + 2m$ excess edges and $3m$ merging times, hence,
$$
|A_*| \leq  4 (g' + 2m)  +   6m.$$ It remains to apply  Proposition \ref{prop:Epath}.
\end{proof}

We are ready to prove Proposition \ref{prop:normR}.

\begin{proof}[Proof of Proposition \ref{prop:normR}]
We define
\begin{equation}\label{eq:choicemR}
m = \left\lceil  \sqrt{ \log n }\right\rceil.
\end{equation}
For this choice of $m$, $n ^{  1 / m } \leq  \exp ( \sqrt{ \log n} )$. Hence, from \eqref{eq:2sum} and Markov inequality, it suffices to prove that for some $c >0$,
\begin{equation}\label{eq:boundSR}
S =  \sum_{s, a, p} | \cW' (s,a,p) | \max_{ \gamma \in W'(s,a,p)}  \PAR{ | \mu'(\gamma) | \sum_{{\stackrel{\gamma'  \in W_{\ell,m} (s,a,p) :}{  \gamma' \sim \gamma}}}     q(\gamma') }  \leq e^{ c m^2 }.
\end{equation}

Let $\gamma \in W'_{\ell,m} (s,a,p)$. Set $g'  = g'(s,a,p) = a - s + p $, by Lemma \ref{le:sumpathR} and Lemma \ref{le:meanpathR},
$$| \mu'(\gamma) | \sum_{{\stackrel{\gamma'  \in W'_{\ell,m} (s,a,p) :}{  \gamma' \sim \gamma}}}     q(\gamma') \leq    d  ^{p}   c^{m+g'}   n^{-g'}.
$$

Now, by Lemma \ref{prop:enumpathR}, since $a \leq 2 \ell m$, $s \leq 2 \ell m + 1 \leq 3 \ell m$, for some new constant $c >1$ changing from line to line,
\begin{eqnarray*}
S & \leq &     \sum_{ s , a, p :  g' (s,a,p) \geq 0, p \leq 2g' (s,a,p) + 10 m} 2^{4m p } ((a +2m)^2 s^2 \ell  )^{4 m ( g'+4)}   d  ^{p} n^{-g'}   c^{m+g'}     \\
 & \leq &  c^m (c \ell m)^{80m}   \sum_{s , g' , p : g' \geq 0, p \leq 2g' + 10 m}  2^{4m p } (c \ell  m  )^{20 m g' } d^{p}  n^{-g'}  ,
 \end{eqnarray*}
 where at the last line, we have performed the change of variable $a \to  g' = a +p -s$. Then, we may sum over $p$, using $(\log n)^c = e^{ o (m)}$ and $ d \leq e^{ m}$, we get for some new constant $c >0$,
\begin{eqnarray*}
S & \leq &  e^{c m^2}  \sum_{s, g' \geq 0}   \PAR{\frac{ L  }{n} }^{g'},
 \end{eqnarray*}
where we have set $L = (c \ell  m)^{20m}$. Since $s \leq 3 \ell m = e^{o(m)}$ and $L/n = o(1)$, we deduce that \eqref{eq:boundSR} holds.
\end{proof}

\section{Proof of Theorem \ref{th:main2}}
\label{sec:end}

All ingredients are finally gathered to prove Theorem \ref{th:main2}.  We start by reducing the range of $\ell$ and $d$ where there is something to be proven. Up to adjusting the final constant $c_1$, we may assume without loss of generality that $d \leq \exp(\sqrt{ \log n})$ and $\ell \leq \log n / \log d$  (otherwise the probabilistic bound is larger than $1$). We fix any $0 < c_0 < c'_0 < \delta$.  Then by Lemma \ref{le:tangle} and Lemma \ref{le:decompBl}, if $\Omega$ is the event that $G$ is $\ell$-tangle free, for any $c > 0$,
\begin{eqnarray*}
\dP \PAR{ \| P^\ell_{ | \IND ^\perp }  \|  \geq  e^{ c \sqrt{ \log n} } \rho^\ell } & = &\dP \PAR{\| P^\ell_{ | \IND ^\perp }  \|  \geq  e^{ c \sqrt{ \log n}   }\rho^\ell ; \Omega} + O( d^{\ell + 2h} n^{-c '_0}) \\
& \leq & \dP \PAR{ J   \geq  e^{ c \sqrt{ \log n} }  \rho^\ell } + O(d^{\ell+2 h} n^{-c'_0}),
\end{eqnarray*}
where $$J = \|  \uP^{(\ell)} \|   +  \frac 1{n}   \sum_{k = 1}^\ell \| R^{(\ell)}_{k} \|.$$
On the other end, by Propositions \ref{prop:normDelta}-\ref{prop:normR}, for some $c'_1 > 0$, with probability at least $1 - 2 n^{-c'_0}$,
\begin{align*}
J  &\leq e^{ c'_1 \sqrt{ \log n}   }\rho^\ell  +   \frac 1 n \sum_{k = 1}^\ell  e^{ c'_1 \sqrt{ \log n}   } \\
&\leq \PAR{ e^{ c'_1 \sqrt{ \log n}   } + \ell e^{\frac{\ell}{2} \log d -\log n} } \rho^\ell,
\end{align*}
where we have used  $\rho \geq 1/ \sqrt d$ by \eqref{eq:lbQHS}. Since $\ell \leq \log n / \log d$, we find that the event
\begin{align*}
J  &\leq \PAR{  e^{ c'_1 \sqrt{ \log n} }  + \frac{ \ell}{\sqrt n}    }\rho^\ell.
\end{align*}
has probability at least  $1 - 2 n^{-c'_0}$. We take any $c > c'_1$ and it remains to adjust the final constant $c_1 > c$ to deal with bounded values of $n$. It concludes the proof of Theorem \ref{th:main2}.

\begin{remark}\label{rk:bunif}
Lemma \ref{le:tangle} and Proposition \ref{prop:Epath} are the only properties of the uniform measures on $\mathbb S_n$ which have been used in the proof. Proposition \ref{prop:Epath} is used in Lemma \ref{le:meanpath} and Lemma \ref{le:meanpathR} where we use that the number of inconsistent arcs is at most $c (g + m)$. The proof may thus be extended to other probability measures on $\mathbb S_n$ with other notions of inconsistency. For example, if $n$ is even, the set of matching $\mathbb M_n$ is the subset of permutations $\sigma \in \mathbb S_n$ such that $\sigma(x) \ne x$ and $\sigma^2 (x) = x$ for all $x \in [n]$. Following \cite{bordenaveCAT}, analogs of Lemma \ref{le:tangle} and Proposition \ref{prop:Epath} hold for the uniform measure on $\mathbb M_n$ (the definition of a consistent arc is slightly more constrained for matchings, but in Lemma \ref{le:meanpath} and Lemma \ref{le:meanpathR}, we may still upper bound the number of inconsistent arcs by $c(m+g)$).
\end{remark}

\begin{remark}\label{rk:bbis}
Proposition \ref{prop:normDelta} and Proposition \ref{prop:normR} are true beyond bistochastic matrices. An inspection of the proof reveals that they hold for any matrix $Q$ provided that $\max_{x} \sum_{y} |Q_{xy}| \leq c$ for some constant $c >0$ (which will have an influence on all other constants).
\end{remark}

\section{Proof of corollaries}

\subsection{Proof of Theorem \ref{th:RWDRW}}

By construction, we have
$
Q_{xy} =  \IND ( (x,y) \in E) / r.
$
It follows that
\begin{equation}\label{eq:qiohs}
\|Q\|_{1 \to \infty} = \frac 1 r \AND \NRMHS{Q} = \frac 1 {\sqrt{r}}.
\end{equation}
It remains to apply Theorem \ref{th:main} with $\delta = 1$.

\subsection{Proof of Corollary \ref{cor:SRDG}}

Let $\cP$ be the set of bi-stochastic matrices of size $n$ with entries in $\{0,1/r\}$.  From the proof of Theorem \ref{th:RWDRW}, for any $Q \in \cP$, \eqref{eq:qiohs} holds.
Note that $A = M B$ for some permutation matrix $M$ is equivalent to $M^* A= B$. It follows that for any permutation matrix $M$,  if $P$ is uniformly  sampled over $\cP$, $P$ and $MP$ have the same distribution. In particular,  $P$ and $MP$ have the same distribution for $M$ uniformly distributed and independent of $P$. We may thus apply Theorem  \ref{th:RWDRW} to $MP$ by conditionning on the value of $P$.


\subsection{Proof of Theorem \ref{th:wperm}}

Up to increasing the constant $c_1$, we may assume that $r \leq \exp ( \sqrt{\log n} )$.  Obviously, if $x \notin S$,
$$
\max_{y} Q_{xy} = \max_i p_i \leq \sqrt{\sum_i p_i^2}.
$$
From our assumption on $S$, it follows that $\NRM{ Q }_{1 \to \infty}^{(\delta)} \leq \sqrt{\sum_i p_i^2}$.

Moreover, we have
$$
Q^\intercal Q = \sum_{i, j} p_i p_j M_i^{*} M_j = \sum_{i} p_i^2 I + \sum_{j \ne i}  p_i p_j M_i^{*} M_j.
$$
From the triangle inequality, we deduce that
$$
\NRMHS{Q} \leq \NRMHS{\sum_{i}p_i^2 I} + \sum_{j \ne i}  p_i p_j \NRMHS{M_i^{*} M_j} = \sqrt{\sum_i p_i^2 } + \sum_{i \ne j} p_i p_j \sqrt{ \frac{1}{n} \sum_{x=1}^n \IND ( \sigma_i (x) = \sigma_j(x))}.
$$
It follows that $\NRMHS{Q} \leq \rho +  \sqrt{|S| / n}  \leq (1 + r^{1/2} n^{-\delta/2} ) \rho$ (where we have used $\sum_i p_i = 1$ and $\sum_i p_i^2 \geq 1/r$). It remains to apply Theorem \ref{th:main}.
\subsection{Proof of Corollary \ref{cor:ani}}

Let $0 < c_0 < 1$ and fix some $c_0 < \delta <1$.  Up to increasing the constant $c_1$, we may assume that $r \leq \exp ( \sqrt{\log n} )$. For any permutation matrix $M$, $P$ has the same distribution than $ MP$. In particular,  $P$ and $MP$ have the same distribution for $M$ uniformly distributed and independent of $M_1, \ldots, M_r$.
Now, let $S = \{ x \in [n] : \exists i \ne j, \sigma_i (x) = \sigma_j (x)\}$. From the union bound, we have
$$
\dE |S| \leq r (r-1) \dP ( \sigma_1 (x) = \sigma_2(x) ) = \frac{r (r-1)}{n}.
$$
Hence, from Markov inequality,
$$
\dP ( |S| \geq n^{1 - \delta}) \leq r^2 n^{\delta-2}.
$$
Finally, on the event $\{ S < n^ {1 - \delta} \}$, we apply Theorem  \ref{th:wperm} for $MP$ by conditioning on the value of $P$.

\subsection{Proof of Theorem \ref{cor3}}

Let $n\geq 1/ r$. From the definition of the transition matrix $Q$ of $f$, and the hypothesis that $f$ maps $I_{k}$ fully to $[0,1]$, it implies that
\begin{align*}
\NRMHS{Q}=\sqrt{ \frac{1}{n} \sum_{x,y} |Q_{xy}|^{2} }=\sqrt{\frac{1}{n}\cdot n \sum_{k=1}^{r} r^{2}}=\frac{1}{\sqrt{r}}.
\end{align*}
Note also that for all $x \in [n]$,  $\max_{y}Q_{yx} \leq \rho$. Similarly, we  get
$$
\NRM{ Q^\intercal Q}_{1 \to 0} = \max_{x \in [n]} | \{ x' : \exists y , Q_{xy} Q_{x' y} \ne 0 \} \leq (1 / r) ^2  .
$$
Finally, we apply Theorem \ref{th:main} with $\delta = 1$ and use that the second largest eigenvalue of $MQ$ in absolute value is equal to $\tau_{f \circ \bar  \sigma}$, whenever $\tau^{\Upsilon}_{f\circ \bar\sigma}\geq(1+\epsilon)\rho>1/r\geq r_{e}(ess_{\Upsilon}(\mathcal{L}_{f}))$.

\bibliographystyle{abbrv}
\bibliography{bib}

\bigskip
\noindent
 Charles Bordenave \\
 Institut de Math\'ematiques de Marseille. CNRS and Aix-Marseille University. \\
39 Rue Fr\'ed\'eric Joliot Curie, 13013 Marseille,  France. \\
\noindent
{E-mail:}\href{mailto:charles.bordenave@univ-amu.fr}{charles.bordenave@univ-amu.fr} \\

\noindent
Yanqi Qiu \\
Institute of Mathematics and Hua Loo-Keng Key Laboratory of Mathematics, AMSS, Chinese Academy of Sciences, Beijing 100190, China;
\\
CNRS, Institut de Math\'ematiques de Toulouse and University of Toulouse III. \\
\noindent
{E-mail:}\href{mailto:yanqi.qiu@amss.ac.cn}{yanqi.qiu@amss.ac.cn} \\

\noindent
Yiwei Zhang \\
School of Mathematics and Statistics, Center for Mathematical Sciences, Hubei Key Laboratory of Engineering Modeling and Scientific Computing, Huazhong University of Sciences and Technology, Wuhan 430074,  China.\\
\noindent
{E-mail:}\href{mailto:yiweizhang@hust.edu.cn}{yiweizhang@hust.edu.cn}

\end{document}